\documentclass[a4paper,11pt]{amsart}
\usepackage[foot]{amsaddr}
\usepackage[margin=2.7cm]{geometry}
\usepackage{graphicx,color, amsfonts,amsmath}
\usepackage{amssymb}
\usepackage[titletoc]{appendix}

\usepackage{systeme}
\usepackage{graphicx}
\usepackage[]{units}%
\usepackage{color}%
\usepackage{mathrsfs}
\usepackage{bm}
\usepackage{hyperref}
\usepackage{float}
\usepackage{amsmath,amsfonts}
\usepackage{isomath}
\usepackage{amsthm}
\usepackage{amssymb}
\usepackage{algorithm,algorithmic}
\usepackage{array}
\usepackage{booktabs}
\usepackage{multicol}
\usepackage{multirow}
\usepackage{tabularx}
\usepackage{subfigure}
\usepackage{cite} 
\usepackage{stackrel}

\usepackage{graphicx}
\newtheorem{thm}{Theorem}[section]
\newtheorem{cor}[thm]{Corollary}
\newtheorem{lem}[thm]{Lemma}

\theoremstyle{definition}

\theoremstyle{remark}
\newtheorem{rem}[thm]{Remark}
\numberwithin{equation}{section}

\newcommand{\vertiii}[1]{{\left\vert\kern-0.25ex\left\vert\kern-0.25ex\left\vert #1 
\right\vert\kern-0.25ex\right\vert\kern-0.25ex\right\vert}}

\newcommand{\blue}{\color{black}} 
\newcommand{\dv}{\color{black}}
\begin{document}
\title[Equal higher order analysis of an unfitted dG method for Stokes systems]{Equal higher order analysis of an unfitted discontinuous Galerkin method for Stokes flow systems}%
\author{Aikaterini Aretaki\textsuperscript{1}}
\address{\textsuperscript{1}Department of Mathematics, National Technical University of Athens, Greece.}
\thanks{}
\email{kathy@mail.ntua.gr}
\author{Efthymios N. Karatzas\textsuperscript{1,2,3}}
\address{\textsuperscript{2}FORTH Institute of Applied and Computational Mathematics, Heraclion, Crete, Greece.}
\address{\textsuperscript{3}SISSA (affiliation), International School for Advanced Studies, Mathematics Area, mathLab Trieste, Italy.}
\email{karmakis@math.ntua.gr \& efthymios.karatzas@sissa.it}

\author{Georgios Katsouleas\textsuperscript{1}}
\thanks{}
\email{gekats@mail.ntua.gr}
\thanks{}%
\subjclass[2000]{Primary}%
\date{\today} 

\thanks{}%

\begin{abstract}
In this work, we analyze an unfitted discontinuous Galerkin discretization for the numerical solution of the Stokes system based on equal higher-order discontinuous velocities and pressures. This approach combines the best from both worlds, firstly the advantages of a piece-wise discontinuous high--order accurate approximation and secondly the advantages of an unfitted to the true geometry grid around possibly complex objects and/or geometrical deformations. Utilizing a fictitious domain framework, the physical domain of interest is embedded in an unfitted background mesh and the geometrically unfitted discretization is built upon symmetric interior penalty discontinuous Galerkin formulation. {\blue To enhance stability we enrich the discrete variational formulation  with a pressure stabilization term. Moreover, the present contribution adopts high order ghost penalty strategies to address the ill conditioning of the system matrix caused by small truncated elements with respect to the unfitted boundary. Motivated by continuous unfitted FEM \cite{BuHa14_III,MLLR12,MLLR14} along with other unfitted mesh surveys grounded on discontinuous spaces \cite{BBH09,GM19,GSM20,M17}, we use proper velocity and pressure ghost penalties defined on faces of cut cells to establish a robust high-order method, in spite of the cell agglomeration technique usually applied on dG methods.}
The current presentation should prove valuable in engineering applications where special emphasis is placed on the optimal effective approximation attaining much smaller relative errors in coarser meshes. Inf-sup stability, the optimal order of convergence, and the {\blue condition number sensitivity with respect to cut configuration} are investigated. {
Numerical examples verify the theoretical results.}
\end{abstract}
\maketitle
\section{Introduction}

The overall objective of this paper is to discuss the discontinuous Galerkin method in an unfitted mesh framework. The prominence both of fictitious domain methods, as well as discontinuous Galerkin methods, is easily explained by their relative advantages. Regarding the former, many practical engineering applications involve problems defined in complex domains whose boundary can even be exposed to large topological changes or deformations. Such cases pose severe challenges in the discretization and even result to  simulations of diminished quality. For instance, the generation of  a suitable conforming mesh is a challenging and computationally intensive task. As a means to bypass such complications, it is instructive to consider the actual computational domain of interest  as being embedded in an unfitted background mesh. More precisely, this can be achieved usually via a geometric parametrization of its boundary via level-set geometries, using a fixed Cartesian background and its associated mesh for each new domain configuration. This approach avoids the need to remesh, as well as the need to develop a reference domain formulation in many applications and methodologies, as typically done in fitted grid FEMs. 


The discontinuous Galerkin method is a robust finite element method that is very well suited to handling complicated geometries with unfitted to the true geometry and/or unstructured meshes.
DG methods generalize the continuous finite element framework by relaxing the continuity constraints at inter-element boundaries, thus providing the tools to manipulate potential jumps via
numerical fluxes \cite{DiPE12}. Such an approach results in additional flexibility in the design of shape functions and enables the use of different polynomial degrees of approximation on adjacent elements, as well as incorporates interfaces between non matching grids and evolving domains \cite{ACLY19,AFRV16,BEFI11,ERW16,GRW05}. Hence, the main motivation for using dG methods in fluid flow problems lies in their robustness in convection-dominated regimes, their conservation properties, and their great flexibility in the mesh-design. Since less communication is required between neighbouring mesh cells, the method is more amenable in parallel computing \cite{LLANC11,SM17} and it is highly attractive in $hp$-adaptive strategies \cite{AGH13,Georgoulis17,SST03,T02} where mesh refinement can be achieved without the continuity restrictions customary in standard finite element methods. 

A higher order analysis of an unfitted discontinuous approach for the Stokes system combines the best of the two methodologies and shows better stability properties than continuous Galerkin, allowing high--order accurate approximations within a geometrically unfitted setting. A major challenge in the unfitted mesh case is that stability and approximations properties as well as the conditioning of the system matrix can be severely impacted by the presence of small cut elements. A possible remedy is to introduce a stabilization term, the so-called ghost penalty term, which prevents the ill-conditioning of the discrete problem. So far, there are few examples in the literature of high-order stabilized unfitted FEM where Stokes equation is considered. In \cite{BuHa14_III,GO18} unfitted finite element pressure--velocity couplings for the Stokes problem are employed, in \cite{JLL15} the authors utilize high--order piecewise polynomials to develop a cut finite element method on composite meshes, while \cite{LPWL16} is based on an isoparametric mapping reconstruction for high accurate geometry approximations.
On the other hand, unfitted dG have mostly been relied on cell agglomeration to deal with the small cut element problem. For instance, in \cite{BDE20} a hybrid high-order (HHO) method has been recently designed and analyzed to approximate the Stokes interface problem on unfitted meshes and in \cite{MKKO16} the compressible Navier–Stokes equations.

Regardless to the cell merging approach, the present work aims to advocate ghost--penalty--type techniques in a higher order analysis of an unfitted dG setting for the Stokes system. In this respect, we augment the discrete system with additional boundary zone ghost penalty terms for both velocity and pressure fields so as to circumvent small cut configuration problems. These terms act on the jumps of the normal  derivatives at faces associated with cut elements and ensure that the condition number is uniformly bounded independently of how the boundary intersects the mesh. Moreover, a fully stabilized scheme is guaranteed by penalizing the pressure jumps across interfaces. To the authors' best knowledge some original introduction of pressure jump/ghost penalties for the Stokes system has only been provided in \cite{BBH09,BCM15} and also some recent contributions on cut dG in \cite{GM19,GSM20,M17}.

Various stabilized finite elements for the Stokes system on fictitious domains have been analyzed in \cite{BaBo04,BCM15,BuHa14_III,GR07,MLLR12,MLLR14} and also extended to the Stokes interface problem in \cite{HaLaZa19,KGR16}. A number of different face-based ghost-penalty stabilizations on cut meshes combined with the continuous interior penalty method has been elaborated in \cite{BFH06,MSW18,WSMW17} and in \cite{SW14} for the transient convection--dominant incompressible Navier--Stokes equations. Analogous work for elliptic boundary value and interface problems has been carried out in \cite{BH2010,BH2012}. 
On the other hand, unfitted dG benefits from the favorable conservation and stability properties of classical dG to solve several  boundary and interface problems \cite{BEFI11,sto,JL13,WC14}  along with two--phase flows \cite{GH14,HEIB13,KK17,K17,MKKO16,SBV11}. A dG variant utilizing divergence--free vector fields for the velocity and continuous pressure approximations for the Stokes and incompressible Navier--Stokes equations has been studied in \cite{BJK90,KJ98}, respectively. Local Discontinuous Galerkin
(LDG) rationales based on mixed formulations of piecewise  solenoidal polynomial velocities and hybrid pressures have been studied in \cite{CCS06,CG05,CKSS02,CKS07}, and also in \cite{MFH08} under interior penalty formulations. Mixed $hp$-discontinuous Galerkin methods for the Stokes problem with a stabilization term penalizing the pressure jumps have been treated in \cite{SST03,T02}. In a vast and non--exhaustive list in the literature on various dG methods, see also \cite{A82,BS08,BS10,CKS09,HL02,LSYY02} and the references therein.

Fictitious domain methods have a long history, dating back to the pioneering  work of Peskin \cite{P1972} and are currently enjoying great popularity, having been successfully applied to a variety of problems. Several improved variants can be found in the recent literature, including such methods as the ghost--cell finite difference method \cite{WFC13}, cut--cell volume method \cite{PHO16}, immersed  interface \cite{KB18}, ghost fluid \cite{BG14}, shifted boundary methods \cite{MS18}, $\phi$--FEM \cite{DL19}, and CutFEM \cite{AK21,BCHLM2014,BH2010,BH2012,BHLMZ18,BHLZ15,HaHa02,KaKaTra20,L19}, among others. For a comprehensive overview of this research area, the interested reader is referred to the review paper \cite{MI05} and also to the recent book volume \cite{UCL16} based on the proceedings of the UCL Workshop 2016. Considerable impetus for such widespread investigations has been provided by  applications in fluids flow or in the context of reduced order modeling for parametrically--dependent domains \cite{KBR19,KSASR2019,KSNSRa2019,KSNSRb2019}. In such cases, immersed and embedded methods compare favorably to standard FEM, providing simple and efficient schemes for the numerical approximation of PDEs in both cases of static and evolving geometries.

Many unfitted variants of discontinuous Galerkin methods have been proposed in the literature as a competitive approaches for simulations in complex and evolving domains \cite{S15}. One of the  first applications involved an elliptic model problem \cite{BE09}, while elliptic interface problems have been discretized via an \itshape hp \normalfont discontinuous Galerkin method  \cite{M12}, an extension of the local dG method \cite{WC14}, and a high--order hybridizable dG method \cite{DWXW17,HNPK13}. In fact, an unfitted dG method was shown to compare favorably to standard dG--FEM \cite{NWBE16}, providing a flexible and accurate alternative to solve the electroencephalography forward problem.  Moreover, we  refer to Saye's important work \cite{S15} in which a numerical quadrature algorithm has been applied to a high-order embedded boundary dG method on curved domains and also to \cite{S17,S17II} for a high-order accurate implicit mesh dG to facilitate precise computation of interfacial fluid flows in evolving geometries. An extension to a parabolic test case has been presented in \cite{BEFI11}. More recently, motivated from PDEs arising from conservation laws on evolving surfaces, an unfitted dG approach was developed for advection problems \cite{ERW16}. Other applications include the linear transport equation \cite{EMNS19}, the Laplace--Beltrami operator on surfaces \cite{BuHaLaMa17} and mixed--dimensional, coupled bulk--surface problems \cite{M17}. In the context of Stokes problems with void or material interfaces, previous efforts include an eXtended hybridizable dG (X-HDG) method \cite{GKK19} combining the hybridizable dG method with an eXtended finite element strategy, considering heaviside enrichment on cut faces/elements.

Our paper is organized as follows. We start with the Stokes flow  model problem and the necessary preliminaries in Section \ref{section2}. The various components of the stabilized unfitted discontinuous Galerkin discretization based on equal higher order discontinuous velocities and pressures are discussed in subsection \ref{22} in detail. Approximation results needed for the analysis of the method are collected in Section \ref{section3}. Section \ref{section4} is devoted to stability estimates and the derivation of the discrete inf--sup condition, followed by a--priori error estimates in Section \ref{section5}. Our theoretical analysis of the method is completed in Section \ref{conditioning}, showing that the condition number of  the stiffness matrix is uniformly  bounded, independently of how the background mesh cuts the boundary. The paper concludes with some numerical tests in Section \ref{section6} which verify the theoretical convergence rates, 
the accuracy and  the geometrical robustness
of the method. 

\begin{figure}
\centering
\includegraphics[scale=0.55]{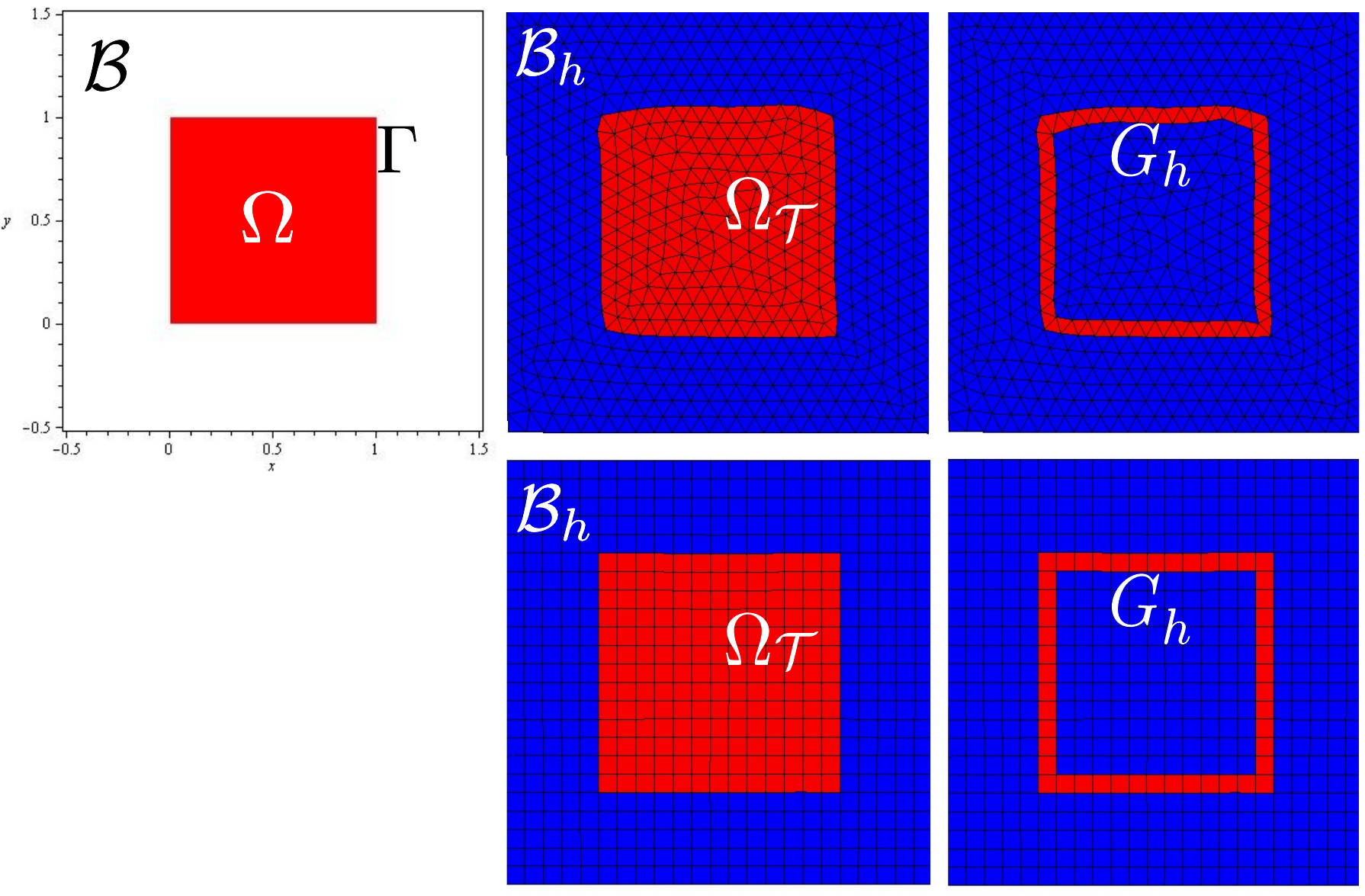}
\caption{The original square domain $\mathrm{\Omega}$ (left picture) and its boundary $\mathrm{\Gamma}$ are represented implicitly
by the level-set function $\phi(x,y)$ in (\ref{lset}) and they are designated by the red colored area. The extended computational domain $\mathrm{\Omega}_{\mathcal{T}}$ is visualized in the middle picture, and it
is covered by the active part of the background mesh $\mathcal{B}_h$ colored in red. The subset $G_h$ of elements in $\mathcal{B}_h$ that intersect the boundary $\mathrm{\Gamma}$ is shown in red at the right picture.}\label{mesh}
\end{figure}

\section{The model  problem and preliminaries} \label{section2}
\subsection{Problem formulation}  
The steady Stokes equations for an incompressible viscous fluid confined in an open, bounded domain $\mathrm{\Omega} \subset \mathbb{R}^d$ ($d=2, 3$) with Lipschitz boundary $\mathrm{\Gamma}=\partial \mathrm{\Omega}$  can be expressed in the form
\begin{eqnarray} \label{The Stokes Problem} 
\nonumber
- \Delta  \mathbf{u}+\nabla p&=&\mathbf{f}\quad\,\,\,\,\, \textrm{in}\,\,\,\mathrm{\Omega}, \\
\nabla \cdot \mathbf{u}&=&0  \quad\,\,\,\,\, \textrm{in}\,\,\, \mathrm{\Omega}, \\
\nonumber \mathbf{u}&=&0 \qquad \textrm{on}\,\,\,\mathrm{\Gamma}.
\end{eqnarray}
Here $\mathbf{u}=(u_1, \dots, u_d):\mathrm{\Omega}  \rightarrow  \mathbb{R}^d$ ($d=2,3$) and $p:\mathrm{\Omega}  \rightarrow  \mathbb{R}$ denote the velocity and pressure fields, and $\mathbf{f} \in \left[L^2(\mathrm{\Omega})\right]^d$ is a forcing term. 
Since the pressure is determined by (\ref{The Stokes Problem}) up to an additive constant, we assume {\blue $\int_{\mathrm{\Omega}}p \, \mathrm{d}x= 0 $} to uniquely determine $p$. Hence, in the following we will consider for pressure the standard space
$$
L_0^2(\mathrm{\Omega}):=\Big\{q \in L^2(\mathrm{\Omega}):\int_{\mathrm{\Omega}}q \,{\blue \mathrm{d}x} =0 \Big\}
$$
of square--integrable functions with zero average over $\mathrm{\Omega}$.

Defining for all $\mathbf{u}, \mathbf{v} \in  V:=[H_0^1(\mathrm{\Omega})]^d$ and $p\in Q:= L_0^2(\mathrm{\Omega})$  the bilinear forms 
\begin{equation}\label{forms_cont}
a(\mathbf{u}, \mathbf{v})= \int_{\mathrm{\Omega}}\nabla \mathbf{u} : \nabla \mathbf{v} \,{\blue \mathrm{d}\mathbf{x}},  \quad 
b(\mathbf{v}, p)= -\int_{\mathrm{\Omega}}p  \nabla \cdot \mathbf{v}\,{\blue \mathrm{d}\mathbf{x}}, 
\end{equation}  
a weak solution to (\ref{The Stokes Problem}) is a pair   $(\mathbf{u}, p) \in [H_0^1(\mathrm{\Omega})]^d \times L_0^2(\mathrm{\Omega})=V\times Q$, such that
\begin{equation}\label{weak}
A(\mathbf{u}, p ; \mathbf{v}, q)=\int_{\mathrm{\Omega}}\mathbf{f} \cdot \mathbf{v}\,{\blue \mathrm{d}\mathbf{x}}, \ \ \textrm{for all test functions} \ \   (\mathbf{v}, q) \in V\times Q,
\end{equation}
with 
$$
A(\mathbf{u}, p ; \mathbf{v}, q)=  a(\mathbf{u}, \mathbf{v}) +b(\mathbf{u}, q)+b(\mathbf{v}, p).
$$
The well--posedness of (\ref{weak}) is standard \cite{DiPE12}.

\subsection{Discretization via an unfitted discontinuous Galerkin method} \label{22}
Implementation of an unfitted discontinuous Galerkin method for the discretization of  (\ref{weak}) requires a fixed background domain $\mathcal{B}$ which contains $\mathrm{\Omega}$. Let $\mathcal{B}_h$ be its corresponding shape--regular mesh, the  \itshape active \normalfont mesh 
$$
\mathcal{T}_h=\left\{T\in \mathcal{B}_h: T\cap \mathrm{\Omega}\neq \emptyset\right\}
$$  
is the  minimal submesh of $ \mathcal{B}_h$ which covers $\mathrm{\Omega}$ and is, in general, \itshape unfitted \normalfont to its boundary $\mathrm{\Gamma}$. As usual, the subscript {\blue $h=\max_{T\in\mathcal{B}_h}h_T=\max_{T\in  \mathcal{B}_h}\mathrm{diam}(T)$} indicates the global mesh size. Finite element spaces for $\mathbf{u}$ and $p$ will be built upon the \itshape extended \normalfont domain $\mathrm{\Omega}_{\mathcal{T}}=\bigcup_{T \in \mathcal{T}_h}T$ which corresponds to  $\mathcal{T}_h$.  The set  of interior faces in the active background mesh is denoted
$$
\mathcal{F}_{h}^{int}=\left\{F=T^{+}\cap T^{-}: T^{+}, T^{-} \in \mathcal{T}_h\right\}.
$$

Fictitious domain  methods, as well as, discontinuous Galerkin related schemes {\blue need Dirichlet boundary conditions at $\mathrm{\Gamma}$ to be weakly satisfied through a variant of Nitsche's method since the mesh does not align with the boundary of the physical domain. Moreover, when Nitsche's approach is applied in the discontinuous Galerkin framework, the continuity of the solution across inter-element boundaries can be attained allowing for independent approximations on different elements and thus, resulting to a consistent discrete scheme \cite{BBH09,GKK19,GM19,GSM20}.}
On the other hand, coercivity over the whole computational domain $\mathrm{\Omega}_{\mathcal{T}}$  is ensured by means of additional ghost penalty terms which act on the gradient jumps in the boundary zone; see, for instance, \cite{BH2012,BuHa14_III,KBR19,MLLR12}. The submesh consisting of all cut elements  is denoted
$$
{G}_h:=\{T\in \mathcal{T}_h:T\cap\mathrm{\Gamma}\neq\emptyset\}
$$
and the relevant set of faces upon which ghost penalty will be applied is given by
$$
\mathcal{F}_{G}:=\left\{F: F \textrm{ is a  face of } T\in G_h, F\, {\blue \nsubseteq}\, \partial\mathrm{\Omega}_{\mathcal{T}}\right\}.
$$
{\blue
We recall that the boundary $\mathrm{\Gamma}$ is well resolved by the mesh $\mathcal{T}_h$ if the following assumptions are satisfied from \cite{BCM15,MLLR12}:} 
{\dv
\begin{itemize}
\item[A:] The intersection between $\mathrm{\Gamma}$ and a facet $F\in\mathcal{F}_h^{int}$ is simply connected; that is, $\mathrm{\Gamma}$ does not cross an interior
facet multiple times.
\item[B:] For each element $T$ intersected by $\mathrm{\Gamma}$, there exists a plane $S_T$ and a piecewise smooth parametrization
$\mathrm{\Phi}: S_T\cap T \to \mathrm{\Gamma}\cap T$.
\item[C:] We assume that there is an integer $N>0$ such that for each element $T\in G_h$, there exists an element $T'\in \mathcal{T}_h\setminus G_h$
and at most $N$ elements $\{T\}_{j=1}^N$ such that $T_1=T$,
$T_N=T'$ and $T_j\cap T_{j+1}\in\mathcal{F}_h^{int}$, $j=1, \ldots, N-1$. In other words, the number of facets to be crossed in order to ``walk" from a cut element $T$ to a non-cut element
$T'\subset\mathrm{\Omega}$ is bounded.
\end{itemize}
}
To define an unfitted discontinuous Galerkin discretization for the Stokes pro\-blem (\ref{weak}), we consider equal--order, elementwise discontinuous  polynomial  finite element pressure and velocity spaces of order $k\geq 1$:
\begin{align*}
V_h &:= \Big\{\mathbf{w}_h\in \Big(L^2(\mathrm{\Omega}_{\mathcal{T}})\Big)^d: \mathbf{w}_h|_{T}\in \Big(\mathcal{P}^{k}(T)\Big)^d, T \in \mathcal {T}_h\Big\} \quad (d=2,3)\\
Q_h &:= \Big\{w_h\in  L^2(\mathrm{\Omega}_{\mathcal{T}}): {\dv\int_\mathrm{\Omega} w_h \, \mathrm{d}x=0}, \,\, w_h|_{T}\in \mathcal{P}^{k}(T), T \in \mathcal {T}_h\Big\}.
\end{align*} 
Moreover, recall the definition 
$$ 
\left\{v\right\}:=\frac{1}{2} \left(v^{+}+v^{-}\right), \quad\quad\left\{\mathbf{v}\right\}:=\frac{1}{2} \left(\mathbf{v}^{+}+\mathbf{v}^{-}\right),
$$
of the {\em average} operator $\left\{\cdot\right\}$ across an interior face $F$ for $v$, $\mathbf{v}$   scalar and vector--valued functions  on $\mathcal{T}_h$ respectively, where $v^{\pm}$ (resp. $\mathbf{v}^{\pm}$) are the traces of $v$ (resp. $\mathbf{v}$) on $F=T^{+}\cap T^{-}$ from the interior of $T^{\pm}$. More precisely, $v^{\pm}(\mathbf{x})=\lim_{t\rightarrow 0^{+}}v(\mathbf{x}{\blue \mp} t{\bf n}_F)$ for $\mathbf{x} \in F$ and  ${\bf n}_F$  the outward--pointing unit normal vector to $F$. The  {\em jump} operator $ [\![\cdot ]\!]$ across $F$  is defined respectively by 
$$ 
[\![v ]\!]:=v^{+}-v^{-}, \quad \quad   [\![\mathbf{v} ]\!]:=\mathbf{v}^{+}-\mathbf{v}^{-}.
$$

With these definitions in place, we are now ready to formulate a discrete counterpart of (\ref{weak}) {\blue employing} an unfitted discontinuous Galerkin method. The symmetric interior penalty discretizations of the diffusion term and the pressure--velocity coupling in (\ref{forms_cont}) lead to the bilinear forms
\begin{align*}
& a_h(\mathbf{u}_h, \mathbf{v}_h) = \int_{\mathrm{\Omega}} \nabla \mathbf{u}_h: \nabla \mathbf{v}_h\,{\blue \mathrm{d}\mathbf{x}} 
- \sum_{F \in \mathcal{F}_h^{int}}\int_{F\cap \mathrm{\Omega}}\left(\left\{  \nabla \mathbf{u}_h\right\}\cdot {\bf n}_F [\![ \mathbf{v}_h  ]\!] + \left\{  \nabla \mathbf{v}_h\right\}\cdot{\bf n}_F [\![ \mathbf{u}_h  ]\!] \right)\,{\blue \mathrm{d}s} \\
& \quad\quad\quad\quad\quad\quad\quad\quad -\int_{\mathrm{\Gamma}}\mathbf{u}_h\nabla \mathbf{v}_h\cdot  \mathbf{n}_{\mathrm{\Gamma}}\,\,{\blue \mathrm{d}s} 
- \int_{\mathrm{\Gamma}}\mathbf{v}_h  \nabla \mathbf{u}_h\cdot \mathbf{n}_{\mathrm{\Gamma}}\,\,{\blue \mathrm{d}s} 
+ \beta h^{-1}\int_{\mathrm{\Gamma}} \mathbf{u}_h\mathbf{v}_h\,{\blue \mathrm{d}s} \\
& \qquad\qquad\quad\quad\quad\quad\quad\quad\quad\quad
\quad\quad\quad\quad\quad\quad\quad\quad
+ \beta h^{-1}\sum_{F\in \mathcal{F}_h^{int}}\int_{F\cap \mathrm{\Omega}}[\![ \mathbf{u}_h  ]\!] [\![ \mathbf{v}_h  ]\!]\,{\blue \mathrm{d}s} , \\
&b_h(\mathbf{v}_h, p_h)= - \int_{\mathrm{\Omega}}p_h\nabla \cdot \mathbf{v}_h \,{\blue \mathrm{d}\mathbf{x}}
+ \sum_{F\in \mathcal{F}_h^{int}}\int_{F\cap \mathrm{\Omega}} [\![ \mathbf{v}_h  ]\!]\cdot \mathbf{n}_F\left\{  p_h\right\}{\blue \mathrm{d}s} 
+\int_{\mathrm{\Gamma}} \mathbf{v}_h \cdot \mathbf{n_{\mathrm{\Gamma}}}  p_h \,{\blue \mathrm{d}s}
\end{align*}
respectively. 
{\dv It is important to mention that in an abuse of notation whenever $\nabla w_h$ is used for functions that lay in the discontinuous Galerkin space, i.e. $w_h\notin H^1(\mathrm{\Omega}_{\mathcal{T}})$, it corresponds to the broken gradient such that $(\nabla w_{h})|_T=\nabla(w_{h}|_T)$ for all $T\in\mathcal{T}_h$. The same applies for the broken divergence operator $\nabla\cdot w_h$ defined element--wise.}

The symmetric interior penalty parameter $\beta>0$ in the definition of $a_h(\cdot, \cdot)$ is chosen sufficiently large  to ensure stability of the method and will be made precise later; see Lemma \ref{a_coerc} and its proof below. For future reference,  note that element--wise integration by parts in the previous forms yields the equivalent formulations
\begin{align}
a_h(\mathbf{u}_h, \mathbf{v}_h) & = -\int_{\mathrm{\Omega}} \Delta \mathbf{u}_h \cdot  \mathbf{v}_h\,{\blue \mathrm{d}\mathbf{x}}
+ \sum_{F \in \mathcal{F}_h^{int}}\int_{F\cap \mathrm{\Omega}} [\![ \nabla \mathbf{u}_h  ]\!] \cdot \mathbf{n}_F \left\{\mathbf{v}_h\right\}\,{\blue \mathrm{d}s} \notag \\
& \quad -\sum_{F \in \mathcal{F}_h^{int}}\int_{F\cap \mathrm{\Omega}} \left\{  \nabla \mathbf{v}_h\right\}\cdot {\bf n}_F [\![ \mathbf{u}_h  ]\!]\,{\blue \mathrm{d}s}
-\int_{\mathrm{\Gamma}}\mathbf{u}_h\nabla \mathbf{v}_h\cdot  \mathbf{n}_{\mathrm{\Gamma}}\,{\blue \mathrm{d}s} \notag \\
& \quad\quad +\beta h^{-1}\int_{\mathrm{\Gamma}} \mathbf{u}_h\mathbf{v}_h\,{\blue \mathrm{d}s} + \beta h^{-1}\sum_{F\in \mathcal{F}_h^{int}}\int_{F\cap \mathrm{\Omega}}[\![ \mathbf{u}_h  ]\!] [\![ \mathbf{v}_h  ]\!]\,{\blue \mathrm{d}s}, \label{a_alt}\\
b_h(\mathbf{v}_h, p_h) & = \int_{\mathrm{\Omega}}\mathbf{v}_h\cdot \nabla p_h\,{\blue \mathrm{d}\mathbf{x}} 
- \sum_{F\in \mathcal{F}_h^{int}}\int_{F\cap \mathrm{\Omega}} \left\{\mathbf{v}_h  \right\}\cdot \mathbf{n}_F [\![ p_h  ]\!]\,{\blue \mathrm{d}s},\label{b_alt}
\end{align}
which will be useful for asserting the consistency of the method.

{\blue 
Owing to the use of equal--order, discontinuous interpolation spaces, the essential inf--sup stability condition is violated. To overcome this difficulty and enhance stability, extra terms need to be added in the dG variational formulation where a standard stabilization involves the pressure face jump penalty
\begin{equation}\label{c}
c_h(p_h, q_h)= \gamma\sum_{F \in \mathcal{F}_h^{int}} \int_{F\cap \mathrm{\Omega}} h_F [\![p_h ]\!][\![q_h  ]\!]\,\mathrm{d}s,
\end{equation}
where $\gamma$ is a positive parameter and $h_F=\min\left\{h_{T}, h_{T'}\right\}$ for $F=T\cap T' \in \mathcal{F}_h^{int}$. 
Some surveys on finite element methods allowing for equal--order velocity and pressure approximations in conjunction with an  interior penalty method can be found, e.g., in \cite{BB06,BB08,CKSS02,CKS09} with an extensive study in the monograph \cite{DiPE12} on dG methods by Di Pietro and Ern.
}

Finally, to extend stabilization on cut elements as well, we also consider the form {\blue \cite{BuHa14_III,MLLR12}}
\begin{equation}\label{j}
J_h(\mathbf{u}_h, p_h; \mathbf{v}_h, q_h)=j_u(\mathbf{u}_h, \mathbf{v}_h)-j_p(p_h, q_h).
\end{equation}
Here, the additional velocity and pressure ghost penalty forms are defined by
\begin{align}
j_u(\mathbf{u}_h, \mathbf{v}_h) & = 
\gamma_{\mathbf{u}}\sum_{j=1}^d\sum_{F\in \mathcal{F}_G}{\blue\sum_{i=0}^k} \int_{F}h_F^{2i-1}[\![ \partial_{\mathbf{n}_F}^i u_{h,j} ]\!] [\![\partial_{\mathbf{n}_F}^i v_{h,j} ]\!]\,{\blue \mathrm{d}s},\label{ju_high_order} \\
j_p(p_h, q_h) & = \gamma_{p}\sum_{F\in \mathcal{F}_G} {\blue \sum_{i=0}^k}\int_{F}h_F^{2i+1}[\![ \partial_{\mathbf{n}_F}^i p_{h} ]\!] [\![ \partial_{\mathbf{n}_F}^i q_{h} ]\!]\,{\blue \mathrm{d}s},\label{jp_high_order}
\end{align}
{\dv where $\partial_{\mathbf{n}}^i v$ is the $i$-th normal derivative given by $\partial_{\mathbf{n}}^i v := \sum_{|\alpha|=i}\frac{1}{\alpha !}D^{\alpha}v(x)\mathbf{n}^{\alpha}$ for multi-index $\alpha=(\alpha_1, \ldots, \alpha_d)$, $|\alpha|=\sum_{i=1}^d\alpha_i$ and $\alpha !=\prod_{i=1}^d\alpha_i!$. Let also $\mathbf{n}^\alpha:=n_1^{\alpha_1}n_2^{\alpha_2}\cdots n_d^{\alpha_d}$ and  $D^{\alpha}:=\frac{\partial^{|\alpha|}}{\partial^{\alpha_1}\partial^{\alpha_2}\cdots\partial^{\alpha_d}}$.}


{\blue The ghost penalty terms defined in (\ref{ju_high_order}) and (\ref{jp_high_order}) are designed to provide sufficient control over the discrete velocity and pressure norms in the extended domain $\mathrm{\Omega}_{\mathcal{T}}$. In particular, when deriving geometrically robust condition numbers, it is critical for the velocity ghost penalty to penalize the lowest order contribution of the form $\int_F h_F^{-1}[\![ \mathbf{u}_h  ]\!][\![ \mathbf{v}_h  ]\!]\,\mathrm{d}s$ and also the higher-order normal derivatives up to the polynomial degree on the entire face in the vicinity of the boundary. The a priori error estimate would still go through with suitably defined discrete velocity norms. Similarly, the same hold for the pressure ghost penalty in (\ref{jp_high_order}). 
For further discussion of a number of issues regarding suitable ghost penalties for dG based discretization, we refer to G{\"u}rkan and Massing in \cite{GM19}.}
The parameters $\gamma_{\mathbf{u}}$ and $\gamma_{p}$ in (\ref{ju_high_order}) and (\ref{jp_high_order})
are positive stabilization constants. More details regarding the CutFEM discretization of the Stokes system can be found in \cite{BuHa14_III}.

Using the previous ingredients, an extended mesh discontinuous Galerkin method for (\ref{weak}) now reads as follows: Find $(\mathbf{u}_h,p_h)\in V_h\times Q_h$, such that 
\begin{equation}\label{cutdg}
A_h(\mathbf{u}_h, p_h ; \mathbf{v}_h, q_h)+J_h(\mathbf{u}_h, p_h ; \mathbf{v}_h, q_h)={\blue L_h(\mathbf{v}_h)}, \ \ \textrm{for all} \ \   (\mathbf{v}_h, q_h) \in V_h\times Q_h.
\end{equation}
The bilinear and linear forms $A_h$ and $L_h$ are defined by
\begin{align}
A_h(\mathbf{u}_h, p_h ; \mathbf{v}_h, q_h) &=  a_h(\mathbf{u}_h, \mathbf{v}_h) +b_h(\mathbf{u}_h, q_h)+b_h(\mathbf{v}_h, p_h)-
{\blue c_h(p_h, q_h)}, \label{A}\\
{\blue L_h(\mathbf{v}_h)} &=  {\blue \int_{\mathrm{\Omega}} \mathbf{f}\cdot  \mathbf{v}_h\, \mathrm{d}\mathbf{x}. \label{L}
}
\end{align}
A similar method without stabilization has been presented in \cite{sto}.

\begin{figure}
\begin{center}
\begin{tabular}{ccc}
\subfigure[$h_0=2^{-2}$]{
\includegraphics[width=0.26\linewidth]{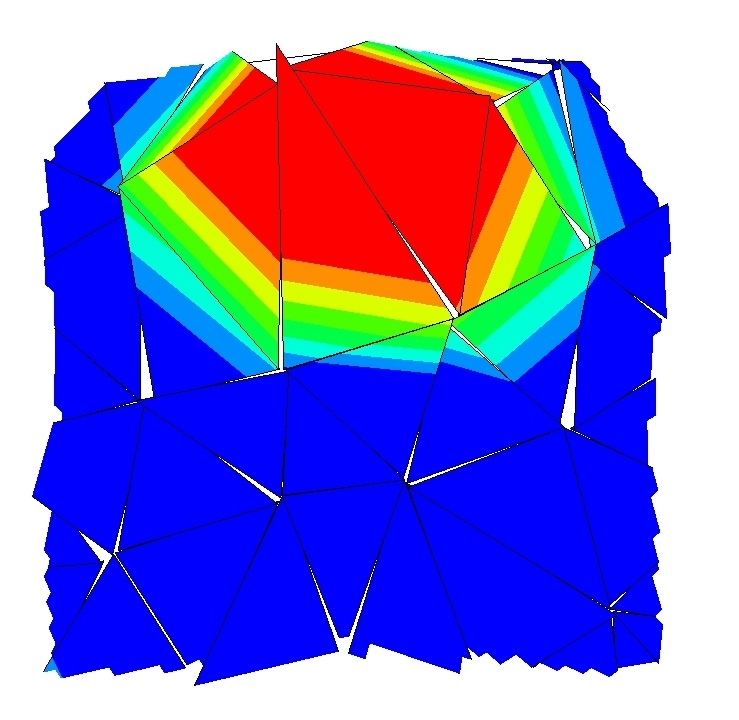}}
\subfigure[$h_1=2^{-3}$]{
\includegraphics[width=0.26\linewidth]{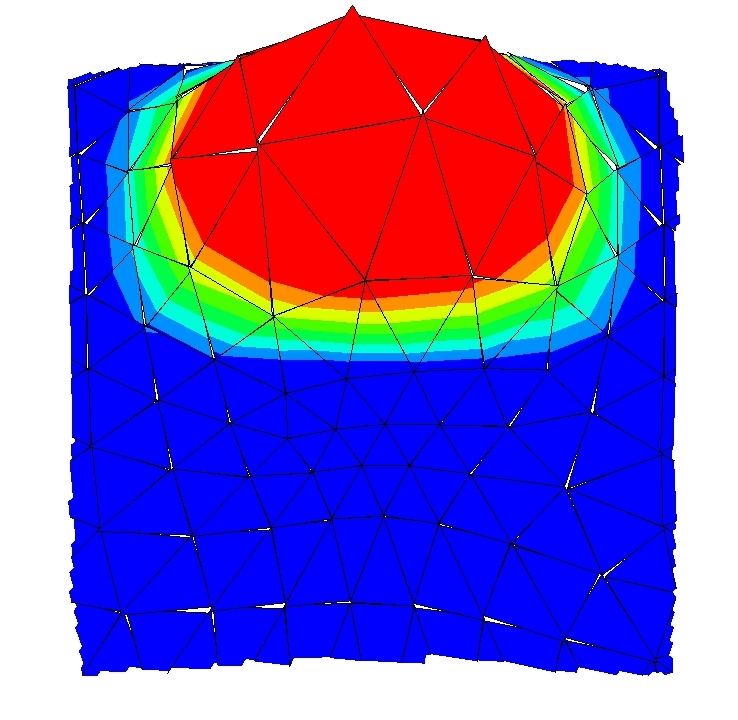}}
\subfigure[$h_2=2^{-4}$]{
\includegraphics[width=0.26\linewidth]{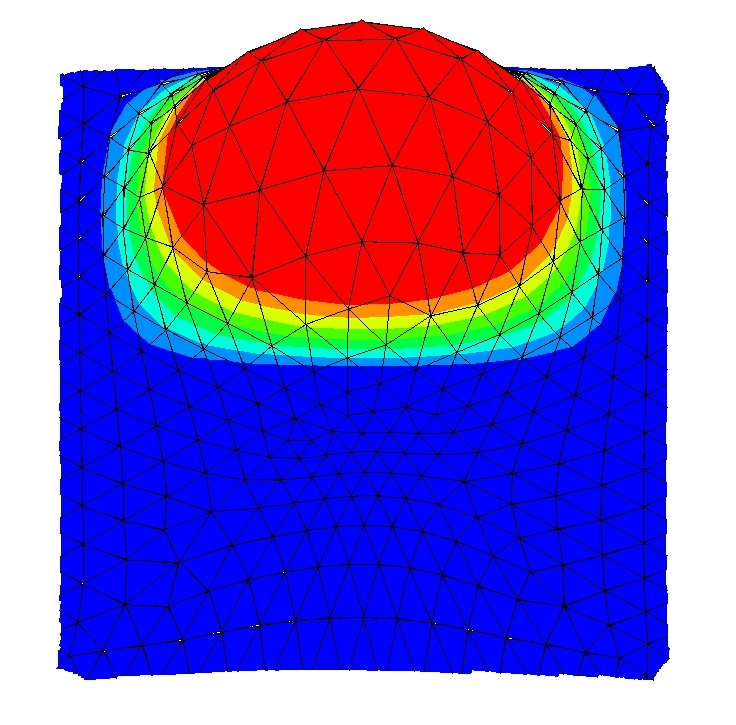}}\\
\subfigure[$h_3=2^{-5}$]{
\includegraphics[width=0.26\linewidth]{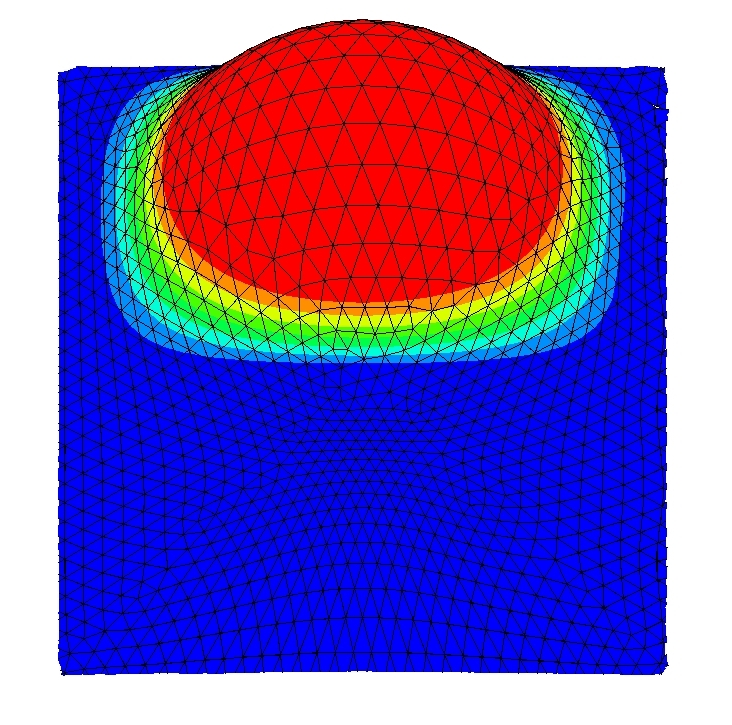}}
\subfigure[$h_4=2^{-6}$]{
\includegraphics[width=0.26\linewidth]{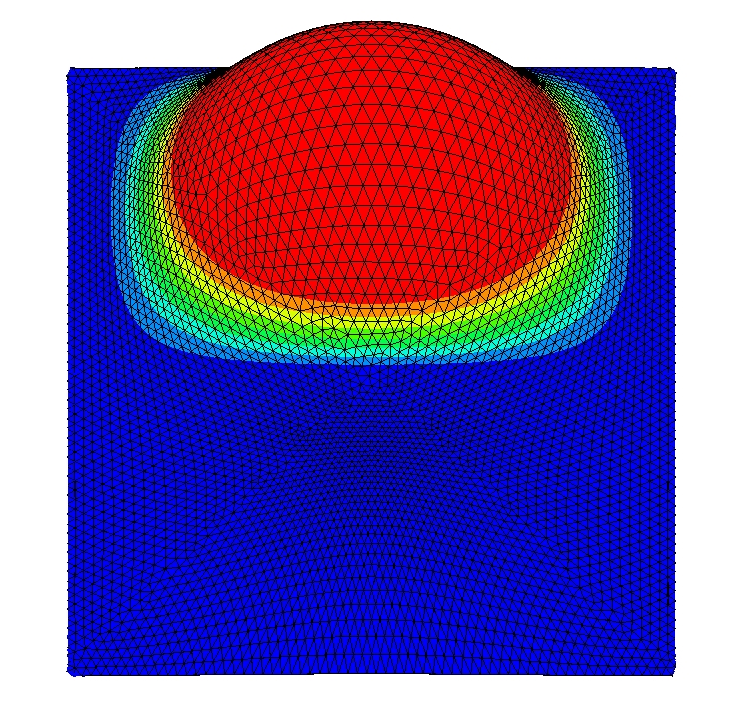}}
\subfigure[$h_5=2^{-7}$]{
\includegraphics[width=0.26\linewidth]{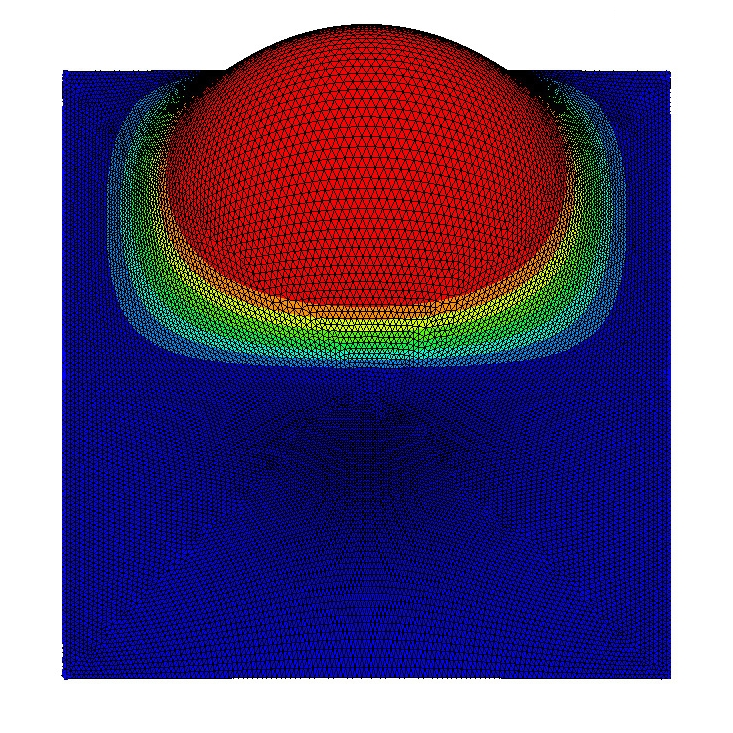}}\\
\end{tabular}
\includegraphics[scale=0.53]{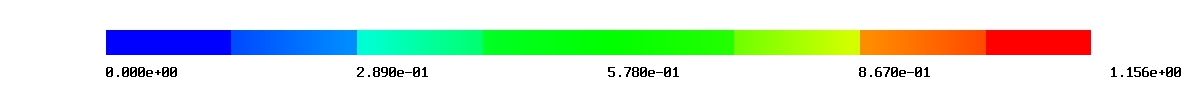}
\caption{Numerical approximation of the first component of the velocity solution visualized for a sequence of successively refined tessellations of the background domain $\mathcal{B}=\left[-0.5, 1.5\right]^{2}$ with mesh parameters $h_{\ell}=2^{-\ell-2}$ ($\ell=0, \dots, 5$) and $P_1-P_1$ finite elements.}\label{coarse}
\end{center}
\end{figure}

\section{Approximation properties} \label{section3}

Throughout this manuscript, standard Sobolev norms and semi--norms on a domain $\mathcal{X}$ for $s\in \mathbb{N}$ will be denoted by $\left\|\cdot\right\|_{s, \mathcal{X}}$ and $|\cdot|_{s, \mathcal{X}}$ respectively, omitting the index in case $s=0$.
A--priori error bounds for the proposed unfitted dG method will be proved with respect to the following mesh--dependent norms:
{\blue
\begin{eqnarray*}
& \vertiii{\mathbf{v}}^2 = \left\|\nabla \mathbf{v}\right\|^2_{\mathrm{\Omega}}+ \left\|h^{-1/2} \mathbf{v}\right\|^2_{\mathrm{\Gamma}} +\left\|h^{1/2}\nabla \mathbf{v}\cdot \mathbf{n}_{\mathrm{\Gamma}}\right\|^2_{\mathrm{\Gamma}} +  \sum_{F\in \mathcal{F}_{h}^{int}} \left\|h^{-1/2}[\![\mathbf{v} ]\!]\right\|^2_{F\cap \mathrm{\Omega}} + \\
&\quad\quad\quad\quad\quad\quad\quad \quad\quad\quad\quad\quad\quad\quad\quad\quad\quad\quad\quad\quad\quad
+ \sum_{T \in \mathcal{T}_h}\left\|h^{1/2}\nabla \mathbf{v} |_{T}\cdot \mathbf{n}_{\partial T}\right\|_{\partial T\cap \mathrm{\Omega}}^2,  \\    
& {\blue \vertiii{p}^2 = \left\|p\right\|^2_{\mathrm{\Omega}}\,\,\,   
+ \left\| h^{1/2}p\right\|^2_{\mathrm{\Gamma}} 
+ \sum_{T \in \mathcal{T}_{h}} \left\|h^{1/2}p\right\|^2_{\partial T\cap \mathrm{\Omega}}
},  \\
& \vertiii{(\mathbf{v},p)}^2 =  \vertiii{\mathbf{v}}^2+\vertiii{p}^2.   
\end{eqnarray*}
}
\normalsize
To investigate stability, we will also make use of the following norms on the extended domain $\mathrm{\Omega}_{\mathcal{T}}$  for the discrete velocity and pressure approximations and their product space:  
\begin{eqnarray}
&\vertiii{\mathbf{v}}^2_{V}= \left\|\nabla \mathbf{v}\right\|^2_{\mathrm{\Omega}_{\mathcal{T}}}+   \left\|h^{-1/2} \mathbf{v}\right\|^2_{\mathrm{\Gamma}} + \sum_{F\in \mathcal{F}_{h}^{int}} \left\|h^{-1/2}[\![\mathbf{v} ]\!]\right\|^2_{F\cap \mathrm{\Omega}} ,\nonumber   \\
& \vertiii{p_h}^2_{Q} =    \left\|  p\right\|^2_{\mathrm{\Omega}_{\mathcal{T}}}\,\,\, \nonumber \\ 
& \vertiii{(\mathbf{v},p)}^2_{V,Q}  =  \vertiii{\mathbf{v}}_{V}^2+\vertiii{p}^2_{Q}. \nonumber 
\end{eqnarray} 
{\dv We should note that the norms $\vertiii{\cdot}$ are defined  on $\mathrm{\Omega}$ and they are used for general functions, while  $\vertiii{\cdot}_\mathcal{X}$ $(\mathcal{X}=V, Q)$ 
represent norms suitable for discrete functions, since they are defined on the extended domain $\mathrm{\Omega}_{\mathcal{T}}$.}

In the following, we summarize certain useful trace inequalities and inverse estimates, which have been proved in \cite{BHLM16,HaHa02,Q09} and will be instrumental in the a--priori error analysis of the method. As in the classical symmetric interior penalty method, the normal flux of a discrete function $v \in \mathcal{P}^k(T)$, $T\in\mathcal{T}_h$ on a face $F\subset\partial T$ or on the boundary $\mathrm{\Gamma}$  is respectively controlled by the inverse inequalities: 
\begin{eqnarray}
\|\partial_{\mathbf{n}_F}^j  v\|_{F} &
\lesssim & h_T^{i-j-1/2} \|D^i v\|_T \quad\forall\,\, T\in\mathcal{T}_h,\,\,\, 0\leq i\leq j, \label{deriv_estimate1} \\
\|\partial_{\mathbf{n}_\mathrm{\Gamma}}^j v\|_{\mathrm{\Gamma}\cap T} & \lesssim & h_T^{i-j-1/2} \|D^i v\|_T \quad \forall\,\, T\in\mathcal{T}_h,\,\,\, 0\leq i\leq j, \label{deriv_estimate2}\\
\|D^j v\|_T  & 
\lesssim &  h_T^{i-j} \|D^i v\|_T\quad\quad\,\,\,\,\forall\,\, T\in\mathcal{T}_h,\,\,\, 0\leq i\leq j \label{deriv_estimate3}
\end{eqnarray}
{\dv where $D^j v$ is the $j$-th total derivative of $v$.}
The notation $a\lesssim b$ (or $a\gtrsim b$) signifies  $a\leq Cb$ (or $a\geq Cb$) for some generic positive constant $C$ that varies with the context, but is always independent of the mesh size and the position of the boundary in relation to the mesh. It is now straightforward to verify that the estimates with respect to the  norms $\vertiii{\cdot }$ and $\vertiii{\cdot}_{\mathcal{X}}$ ($\mathcal{X}=V,Q$) are related via
\begin{equation}\label{norm_est1}
\vertiii{\mathbf{v}}\lesssim \vertiii{\mathbf{v}}_{V}, \quad \vertiii{p}\lesssim \vertiii{p}_{Q},
\end{equation}
{\dv which hold only for discrete functions as a consequence of \eqref{deriv_estimate1}--\eqref{deriv_estimate3}.} {\blue Furthermore, setting $i=j=0$ in \eqref{deriv_estimate1}, \eqref{deriv_estimate2} the next trace inequalities immediately follow for $v\in H^1(\mathrm{\Omega}_{\mathcal{T}})$, \cite{Ha05,HHL03}},
\begin{eqnarray}
\left\|v\right\|_{T \cap \mathrm{\Gamma}} & \lesssim \left(h_T^{-1/2}\left\|v\right\|_{T}+h_T^{1/2}\left\|\nabla v\right\|_{T}\right) \quad \textrm{for} \ T \in \mathcal{T}_h,  \label{tr1a}\\
\left\|v\right\|_{\partial T} &  \lesssim \left(h_T^{-1/2}\left\|v\right\|_{T}+h_T^{1/2}\left\|\nabla v\right\|_{T}\right) \quad \textrm{for} \ T \in \mathcal{T}_h. \label{tr2a}
\end{eqnarray}

Now, the following statement recalls the corresponding definitions and the necessary approximation results for the analysis. {\dv For brevity of presentation, we only state the properties for the scalar--valued pressure space, since they easily extend to the vector--valued velocity space.}

\begin{lem}\label{approx_revised}
Let ${\dv\mathcal{E}^s}: H^s(\mathrm{\Omega}) \rightarrow {\blue H^s(\mathbb{R}^d)}$ ($s\geq0$) be an $H^s$--extension operator on $\mathbb{R}^d$, such that ${\dv\mathcal{E}^s}\phi|_{\mathrm{\Omega}} =\phi|_{\mathrm{\Omega}}$, ${\dv\mathcal{E}^s}\phi|_\mathrm{\Gamma} =\phi|_\mathrm{\Gamma}$, $\left\|{\dv\mathcal{E}^s}\phi\right\|_{s, \mathrm{\Omega}_{\mathcal{T}}}\leq C\left\|\phi\right\|_{s, \mathrm{\Omega}}$ for any $\phi\in H^s(\mathrm{\Omega})$ and {\dv$\Pi_h: L^2(\mathrm{\Omega}) \rightarrow Q_h$} the Scott-Zhang-type extended interpolation operator  defined by 
\begin{equation}\label{interpolation_revised}
{\dv\Pi_h \phi= \Pi^*_h{\dv\mathcal{E}^s} \phi,}
\end{equation}
where {\dv$\Pi_h^*: L^2(\mathrm{\Omega}_{\mathcal{T}}) \rightarrow Q_h$} is the standard Scott-Zhang interpolation.  Then, the estimates
\begin{align}
\left\|v- {\dv\Pi_h} v \right\|_{r,T} &\leq C h^{s-r}_{T} |v|_{s,\Delta_T}, \,\,\, 0\leq r\leq s, \,\,\, \textrm{for every} \,\,\, T \in \mathcal{T}_h, \label{interpolation_est1_revised} \\ 
\!\!\!\!\!\!
\left\|v- {\dv\Pi_h v} \right\|_{r,F} &\leq C h^{s-r-1/2}_{F} |v|_{s,\Delta_F},\,0\leq r\leq s-1/2, \,\textrm{for every}\,F\in\mathcal{F}_h^{int}, \label{interpolation_est2_revised} 
\end{align}
hold for every {\dv $v \in H^s(\mathrm{\Omega}
)$}, where $\Delta_{\mathcal{X}}$ ($\mathcal{X}=T, F$) denotes the corresponding patch of neighbors; i.e., the set of elements sharing at least one vertex with the element $T$ or the element face $F$, respectively.
\end{lem}
Furthermore, the local approximation properties of the extended Scott-Zhang interpolation {\dv$\Pi_h$} along with the stability of the extension operator ${\dv\mathcal{E}^s}$, give rise to the global error estimate
{\blue
\begin{equation}\label{eq22}
\left\|v- \Pi_h v \right\|_{r,\mathrm{\Gamma}} \leq C h^{s-r-1/2} |v|_{s,\mathrm{\Omega}}, 
\quad 0\leq r\leq s-1/2.
\end{equation}
}
{\dv The vector--valued version of the Scott-Zhang extended interpolation operator $\mathbf{\Pi}_h:[L^2(\mathrm{\Omega})]^d\to V_h$ can be constructed analogously to $\Pi_h$ in Lemma \ref{approx_revised}. Apparently, the interpolation operators $\mathbf{\Pi}_h$ and $\Pi_h$ render the same approximation and stability properties. }

In a similar fashion as in \cite{BuHa14_III,MLLR12}, we interpolate a pair   $(\mathbf{u}, p) \in \left[H^{2}(\mathrm{\Omega})\right]^d \times H^1(\mathrm{\Omega})$ through interpolants of $\left[H^{k+1}\right]^d \times H^k$--extensions of the functions $(\mathbf{u}, p)$  on {\blue $\mathbb{R}^d$}. Keeping the same notation of the extension operator for both the velocity and pressure spaces, we choose $\mathcal{E}^s$ as in Lemma \ref{approx_revised} such that $\mathcal{E}^{k+1}\mathbf{u}|_{\mathrm{\Omega}}=\mathbf{u}$ and $\mathcal{E}^{k}p|_{\mathrm{\Omega}}=p$ and interpolation operators  $\mathbf{\Pi}_h:\left[H^{k+1}(\mathrm{\Omega})\right]^d\to V_h$ and $\Pi_h:	H^k(\mathrm{\Omega}) \to Q_h$. Estimates for an interpolation error of the associated interpolants with respect to the $\vertiii{\cdot}$--norm follow in the next result.

\begin{cor}\label{discrete_error_revised}
The  approximation errors of the extended interpolation operators {\dv$\mathbf{\Pi}_h$ and $\Pi_h$} \itshape for $(\mathbf{u}, p) \in \left[H^{k+1}(\mathrm{\Omega})\right]^d \times H^k(\mathrm{\Omega})$ satisfy
\begin{eqnarray}
\vertiii{\mathbf{u}-\mathbf{\Pi}_h \mathbf{u}}&\leq & C h^{k}\left|\mathbf{u}\right|_{k+1,\mathrm{\Omega}}, \label{approx_er1_revised} \\
\vertiii{(\mathbf{u}-\mathbf{\Pi}_h \mathbf{u}, p-{\dv\Pi_h p})}&\leq & C h^{k} \Big(\left|\mathbf{u}\right|_{k+1, \mathrm{\Omega}}+\left|p\right|_{k, \mathrm{\Omega}}\Big). \label{approx_er2_revised} 
\end{eqnarray}
\end{cor}
\begin{proof}
{\dv
It is instructive to introduce the auxiliary norm 
\begin{multline*}
\vertiii{\mathbf{v}}^2_h = \left\|\nabla \mathbf{v}\right\|^2_{\mathrm{\Omega}_{\mathcal{T}}}+ \left\|h^{-1/2} \mathbf{v}\right\|^2_{\mathrm{\Gamma}} +\left\|h^{1/2}\mathbf{n}_{\mathrm{\Gamma}}\cdot \nabla \mathbf{v}\right\|^2_{\mathrm{\Gamma}} +  \sum_{F\in \mathcal{F}_{h}^{int}} \left\|h^{-1/2}[\![\mathbf{v} ]\!]\right\|^2_{F\cap \mathrm{\Omega}} \\
+ \sum_{T \in \mathcal{T}_h}\left\|h^{1/2}\nabla \mathbf{v} |_{T}\cdot \mathbf{n}_{\partial T}\right\|_{\partial T}^2,   
\end{multline*}
which clearly dominates $\vertiii{\mathbf{v}}$, in the sense that
$\vertiii{\mathbf{v}-\mathbf{\Pi}_h\mathbf{v}}\leq 
\vertiii{\mathcal{E}^{k+1}\mathbf{v}-\mathbf{\Pi}_h\mathbf{v}}_h$. 
Hence, it is sufficient to prove the statement for $\vertiii{\cdot}_h$ instead of $\vertiii{\cdot}$. 
Setting $\mathbf{e}_{\pi}=\mathcal{E}^{k+1}\mathbf{u}-\mathbf{\Pi}_h\mathbf{u}$, we have by definition
\begin{align*}
\vertiii{\mathbf{e}_{\pi}}^2_h &= 
\left\|\nabla \mathbf{e}_{\pi}\right\|^2_{\mathrm{\Omega}_{\mathcal{T}}}
+ \left\|h^{-1/2} \mathbf{e}_{\pi}\right\|^2_{\mathrm{\Gamma}} +\left\|h^{1/2}\nabla \mathbf{e}_{\pi}\cdot \mathbf{n}_{\mathrm{\Gamma}}\right\|^2_{\mathrm{\Gamma}}
+\sum_{F\in \mathcal{F}_{h}^{int}} \left\|h^{-1/2}[\![\mathbf{e}_{\pi} ]\!]\right\|^2_{F\cap \mathrm{\Omega}}\\ 
& \qquad\qquad\qquad + \sum_{F\in \mathcal{F}_h^{int}}\left\|h^{1/2}\left\{\nabla\mathbf{e}_{\pi}\right\}\cdot \mathbf{n}_{F}\right\|^2_{F\cap 
\mathrm{\Omega}} \\
&= I_1 +I_2+I_3+I_4+I_5.
\end{align*}
Terms $I_1$ and $I_5$ may be simply estimated, using the local approximation property (\ref{interpolation_est1_revised}), the inverse estimate (\ref{deriv_estimate3}) and the stability  of the extension operator $\mathcal{E}^{k+1}$. For instance, 
\begin{align*}
\left\|\nabla \mathbf{e}_{\pi}\right\|_{\mathrm{\Omega}_{\mathcal{T}}}  = \sum_{T\in\mathcal{T}_h} \left\|\nabla \mathbf{e}_{\pi}\right\|_{T}  &\stackrel{(\ref{deriv_estimate3})}{\lesssim} \sum_{T\in\mathcal{T}_h}h_T^{-1}\left\|\mathbf{e}_{\pi}\right\|_{T} \stackrel{(\ref{interpolation_est1_revised})}{\lesssim} \sum_{T\in\mathcal{T}_h}h_T^k|\mathcal{E}^{k+1}\mathbf{u}|_{k+1,\Delta_T} \\
& \lesssim h^k|\mathcal{E}^{k+1}\mathbf{u}|_{k+1,\mathrm{\Omega}_{\mathcal{T}}}
\lesssim h^k|\mathbf{u}|_{k+1,\mathrm{\Omega}}.
\end{align*}
Proceeding in a similar fashion, $I_2$ and $I_3$ can be treated by applying the global error estimate $\left\|\mathbf{v}- \mathbf{\Pi}_h \mathbf{v} \right\|_{\mathrm{\Gamma}}\lesssim h^{k+1/2} |\mathbf{v}|_{k+1,\mathrm{\Omega}}$ and (\ref{deriv_estimate2}), while estimate (\ref{interpolation_est2_revised}) combined with (\ref{deriv_estimate1}) gives the desired bounds for $I_4$ and  the proof of (\ref{approx_er1_revised}) is complete.

The proof of the estimate (\ref{approx_er2_revised}) for the approximation error in the product space is similar, considering the  auxiliary pressure norm
\begin{equation*}   
{\blue 
\vertiii{p}^2_h =    \left\| p\right\|^2_{\mathrm{\Omega}_{\mathcal{T}}}\,\,\,
+ \left\| h^{1/2}p\right\|^2_{\mathrm{\Gamma}} 
+\sum_{T\in \mathcal{T}_{h}} \left\|h^{1/2}p\right\|^2_{\partial T\cap \mathrm{\Omega}} }
\end{equation*}
and proving the assertion for $\vertiii{\mathcal{E}^kp-\Pi_hp}_h$.
}
\end{proof}

To prove the stability of the method, we will also need a continuity property for $\mathbf{\Pi}_h$  with respect to different norms. 
{\blue
\begin{lem}\label{interpolation_continuity_revised}
The vector-valued extended interpolation operator $\mathbf{\Pi}_h$ satisfies
\begin{equation}\label{interp_est_revised}
\vertiii{ \mathbf{\Pi}_h \mathbf{v} }_{V}\leq C_{\Pi} \left\|\mathbf{v}\right\|_{1,\mathrm{\Omega}},\,\,\, \textrm{for every}\,\,\,\mathbf{v} \in \Big[H_0^1(\mathrm{\Omega})\Big]^d,
\end{equation}
for some positive constant $C_\Pi$.
\end{lem} }
{\dv 
\begin{proof}
By definition,
$$
\vertiii{\mathbf{\Pi}_h \mathbf{v}}_{V}= \left\|\nabla  \mathbf{\Pi}_h\mathbf{v}\right\|^2_{\mathrm{\Omega}_{\mathcal{T}}}+   \left\|h^{-1/2}  \mathbf{\Pi}_h\mathbf{v}\right\|^2_{\mathrm{\Gamma}} + \sum_{F\in \mathcal{F}_{h}^{int}} \left\|h^{-1/2}[\![ \mathbf{\Pi}_h\mathbf{v} ]\!]\right\|^2_{F\cap \mathrm{\Omega}}.
$$
The bound for the first term follows directly from the definition of $\mathbf{\Pi}_h$ and the continuity of the extension operator $\mathcal{E}^1$. 
Making use of the trace inequality (\ref{tr1a}) and the fact that $\mathcal{E}^1\mathbf{v}|_{\mathrm{\Gamma}}=0$ for $ \mathbf{v} \in \left[H_0^1(\mathrm{\Omega})\right]^d$, the bound for the second term is evident. The bound for the third term  
\begin{align*}
\sum_{F\in \mathcal{F}_{h}^{int}} \left\|h^{-1/2}[\![ \mathbf{\Pi}_h\mathbf{v} ]\!]\right\|^2_{F\cap \mathrm{\Omega}}
&\leq \sum_{T \in \mathcal{T}_h}\sum_{F \subset \partial T} \left\|h^{-1/2} \left(\mathcal{E}^1\mathbf{v}-\mathbf{\Pi}_h\mathbf{v} \right)\right\|^2_{F\cap\mathrm{\mathrm{\Omega}}}\lesssim \left\|\mathbf{v}\right\|_{1,\mathrm{\Omega}},
\end{align*}
follows as well, due to (\ref{interpolation_est2_revised}).
\end{proof} 
}

\section{Stability estimates} \label{section4}

The fact that the discrete problem is well-posed follows by the inf--sup stability of the bilinear form $A_h+J_h$ in the formulation (\ref{cutdg}) with respect to the $\vertiii{\cdot}_{V,Q}$--norm.  We begin by investigating the properties of the separate forms which contribute to $A_h+J_h$.  

A useful observation is that the form $a_h(\cdot, \cdot)$, augmented by $j_u(\cdot, \cdot)$, is continuous and coercive with respect to the norm $\vertiii{\cdot }_{V}$. For its proof, we will make use of the fact that the ghost penalty term $j_u(\cdot, \cdot)$ extends the control from the physical domain $\mathrm{\Omega}$ to the entire active mesh; i.e., on the extended domain $\mathrm{\Omega}_{\mathcal{T}}$:

\begin{lem}[{\blue\cite{GM19,GSM20}} ]\label{ext}  
There are constants $C_{v}, C_{p}>0$, depending only on the
shape-regularity and the polynomial order and not on the mesh or the location of the boundary, such that the following estimates hold: 
\begin{equation}\label{ext_v}
\left\|\nabla \mathbf{v}_h\right\|^2_{\mathrm{\Omega}_{ \mathcal{T}} }  \leq C_{v} \left(\left\|\nabla \mathbf{v}_h\right\|^2_{\mathrm{\Omega}}+j_u(\mathbf{v}_h, \mathbf{v}_h)\right) \leq C_{v} \left\|\nabla \mathbf{v}_h\right\|^2_{\mathrm{\Omega}_{ \mathcal{T}} },  \,\, \textrm{for all } \mathbf{v}_h \in V_h
\end{equation}
and
\begin{equation}\label{ext_p}
\left\| p_h\right\|^2_{\mathrm{\Omega}_{ \mathcal{T}} }  \leq C_{p} \left(\left\| p_h\right\|^2_{\mathrm{\Omega}}+j_p(p_h, p_h)\right) \leq C_{p} \left\| p_h\right\|^2_{\mathrm{\Omega}_{ \mathcal{T}} }, \,\, \textrm{for all } p_h \in Q_h.
\end{equation}
\end{lem} 

With this preliminary result in place, we are now ready to prove:

\begin{lem}[Discrete coercivity of $a_h+j_u$]\label{a_coerc}
For suitably large  discontinuity penalization   parameter $\beta>0$ in the definition of the bilinear form $a_h(\cdot, \cdot)$, there exists a constant $c_a>0$, such that  
\begin{equation}\label{coercivit}
c_{a} \vertiii{\mathbf{v}_h}_{V}^2 \leq a_h(\mathbf{v}_h, \mathbf{v}_h)+j_u(\mathbf{v}_h, \mathbf{v}_h),\,\,\,\textrm{for any} \,\,\,\mathbf{v}_h \in V_h.
\end{equation} 
\end{lem}
\begin{proof}
The proof follows closely the standard arguments for the usual symmetric interior penalty method. More precisely, for any $\epsilon \in \mathbb{R}_{+}$, we have 
{\blue
\begin{align}
& a_h(\mathbf{v}_h, \mathbf{v}_h)+j_u(\mathbf{v}_h, \mathbf{v}_h) = \left\|\nabla \mathbf{v}_h\right\|^2_{\mathrm{\Omega}}+j_u(\mathbf{v}_h, \mathbf{v}_h) +
\notag\\
&\qquad\qquad\qquad\qquad\qquad+ \beta \Big(\left\|h^{-1/2}\mathbf{v}_h\right\|^2_{\mathrm{\Gamma}}+ \sum_{F\in \mathcal{F}_h^{int}}\left\|h^{-1/2}[\![\mathbf{v}_h ]\!]\right\|^2_{F\cap \mathrm{\Omega}}\Big) -\notag \\
&\qquad\qquad\qquad\qquad\qquad
-2\sum_{F\in \mathcal{F}_h^{int}}\int_{F\cap \mathrm{\Omega}}\left\{\nabla \mathbf{v}_h\right\}\cdot \mathbf{n}_F [\![\mathbf{v}_h ]\!] \,\,\mathrm{d}s - 2\int_{\mathrm{\Gamma}}\mathbf{v}_h  \nabla \mathbf{v}_h\cdot \mathbf{n}_{\mathrm{\Gamma}} \,\,\mathrm{d}s \notag \\
&\geq  \left\|\nabla \mathbf{v}_h\right\|^2_{ \mathrm{\Omega}  }+j_u(\mathbf{v}_h, \mathbf{v}_h)+\beta\Big( \left\|h^{-1/2}\mathbf{v}_h\right\|^2_{\mathrm{\Gamma}}+\sum_{F\in \mathcal{F}_h^{int}}\left\|h^{-1/2}[\![\mathbf{v}_h ]\!]\right\|^2_{F\cap  \mathrm{\Omega}} \Big) - \notag \\
& \qquad\quad\quad\quad\quad
-  \epsilon \sum_{F\in \mathcal{F}_h^{int}}\left\|h^{1/2}\left\{\nabla\mathbf{v}_h\right\}\cdot \mathbf{n}_F\right\|^2_{F\cap \mathrm{\Omega}}- \epsilon^{-1} \sum_{F\in \mathcal{F}_h^{int}}\left\|h^{-1/2}[\![\mathbf{v}_h ]\!]\right\|^2_{F\cap \mathrm{\Omega}} - \notag \\
& \qquad\quad\quad\quad\quad 
-  \epsilon   \left\|h^{1/2}\nabla\mathbf{v}_h\cdot \mathbf{n}_{\mathrm{\Gamma}}\right\|^2_{\mathrm{\Gamma}}-\epsilon^{-1} \left\| h^{-1/2}\mathbf{v}_h \right\|^2_{\mathrm{\Gamma}} \notag \\
& \geq \left\|\nabla \mathbf{v}_h\right\|^2_{ \mathrm{\Omega}  }+j_u(\mathbf{v}_h, \mathbf{v}_h)+(\beta-\epsilon^{-1})\Big( \left\|h^{-1/2}\mathbf{v}_h\right\|^2_{\mathrm{\Gamma}}+ \sum_{F\in \mathcal{F}_h^{int}}\left\|h^{-1/2}[\![\mathbf{v}_h ]\!]\right\|^2_{F\cap \mathrm{\Omega}}\Big)- \notag\\
& \qquad\quad\quad\quad\quad 
-\epsilon\Big(\sum_{F\in \mathcal{F}_h^{int}}\left\|h^{1/2}\left\{\nabla\mathbf{v}_h\right\}\cdot \mathbf{n}_F\right\|^2_{F\cap \mathrm{\Omega}} +  \left\|h^{1/2}\nabla\mathbf{v}_h\cdot \mathbf{n}_{\mathrm{\Gamma}}\right\|^2_{\mathrm{\Gamma}}\Big). \label{coerc}
\end{align}
}
A lower bound for the latter term in (\ref{coerc}) is readily obtained through the inverse estimates (\ref{deriv_estimate1}) and (\ref{deriv_estimate2}).  
In particular, note for $F \in \mathcal{F}_h^{int}$ with $F=\partial T \cap \partial T^{'}$ that 
\begin{align*}
\left\|h^{1/2}\left\{\nabla\mathbf{v}_h\right\}\cdot \mathbf{n}_F\right\|_{F\cap \mathrm{\Omega}}
& \leq \frac{1}{2}\left(\left\|h^{1/2}\nabla\mathbf{v}_h\cdot \mathbf{n}_F\right\|_{F\subset \partial T  }+\left\|h^{1/2}\nabla\mathbf{v}_h\cdot \mathbf{n}_F\right\|_{F\subset \partial T^{'} }\right) \\
& \lesssim \max_{i=T,T^{'}}\left\{\left\|\nabla{\mathbf{v}_h}\right\|_{i}\right\}
\end{align*}
and then summing over all interior faces in the active mesh, we estimate
\begin{equation}\label{par1}
\sum_{F\in \mathcal{F}_h^{int}}\left\|h^{1/2}\left\{\nabla\mathbf{v}_h\right\}\cdot \mathbf{n}_F\right\|^2_{F\cap \mathrm{\Omega}}
\lesssim   \left\|\nabla \mathbf{v}_h\right\|_{\mathrm{\Omega}_{\mathcal{T}}}^2.
\end{equation}
Likewise, using (\ref{deriv_estimate2})
\begin{equation}\label{par2}
\left\|h^{1/2}\nabla\mathbf{v}_h\cdot \mathbf{n}_{\mathrm{\Gamma}}\right\|^2_{\mathrm{\Gamma}} = \sum_{T\cap \mathrm{\Gamma}\neq \emptyset}\left\|h^{1/2}\nabla\mathbf{v}_h\cdot \mathbf{n}_{\mathrm{\Gamma}}\right\|^2_{T\cap \mathrm{\Gamma}}  \lesssim \sum_{T\cap \mathrm{\Gamma}\neq \emptyset}\left\|\nabla\mathbf{v}_h\right\|^2_{T}  \lesssim  \left\|\nabla \mathbf{v}_h\right\|^2_{\mathrm{\Omega}_{\mathcal{T}}}. 
\end{equation}
Then, application of (\ref{ext_v}) verifies, for a suitable   choice of $\epsilon$, that  the terms in (\ref{par1}) and (\ref{par2}) can be  dominated by the leading two terms in (\ref{coerc}). Indeed, letting $\widehat{C}_1$ and  $\widehat{C}_2$ the constants in (\ref{par1}) and (\ref{deriv_estimate2}) respectively and  collecting all estimates, we conclude 
\begin{eqnarray*}
a_h(\mathbf{v}_h, \mathbf{v}_h)+j_u(\mathbf{v}_h, \mathbf{v}_h) 
& \geq & \left(C_{v}^{-1}-\epsilon (\widehat{C}_1+\widehat{C}_2)\right)\left\|\nabla \mathbf{v}_h\right\|^2_{\mathrm{\Omega}_{\mathcal{T}}} \\
& & +(\beta-\epsilon^{-1})\Big( \left\|h^{-1/2}\mathbf{v}_h\right\|^2_{\mathrm{\Gamma}}+ \sum_{F\in \mathcal{F}_h^{int}}\left\|h^{-1/2}[\![\mathbf{v}_h ]\!]\right\|^2_{F\cap \mathrm{\Omega}}\Big).
\end{eqnarray*}
Coercivity (\ref{coercivit}) is already {\blue satisfied}  for  $\beta> \epsilon^{-1}>C_v(\widehat{C}_1+\widehat{C}_2)$. The  corresponding coercivity constant is $c_a=\min\big\{C_{v}^{-1}-\epsilon (\widehat{C}_1+\widehat{C}_2), \beta-\epsilon^{-1}\big\}$.
\end{proof}
{\blue
\begin{lem}[Continuity]\label{a_bound_revised}
Let $V_{*}=[H^{k+1}(\mathrm{\Omega})\cap H^1_0(\mathrm{\Omega})]^d$ and $Q_{*}=H^k(\mathrm{\Omega})\cap L^2_0(\mathrm{\Omega})$. 
{\blue Then there exist constants $C_a , C_b>0$, such that} 
\begin{eqnarray}
[a_h+ j_u](\mathbf{u}_h, \mathbf{v}_h) &\leq & C_a \vertiii{\mathbf{u}_h}_{V}\cdot \vertiii{\mathbf{v}_h}_{V},\quad \forall\,\, \mathbf{u}_h, \mathbf{v}_h \in V_h,\label{cont3_revised}\\
a_h(\mathbf{u}, \mathbf{v}_h) &\leq & C_a \vertiii{\mathbf{u}}\cdot  \vertiii{\mathbf{v}_h}, \,\,\,
\forall\,\, (\mathbf{u}, \mathbf{v}_h) \in(V_*+V_h)\times V_h, \label{cont1_revised}\\
b_h(\mathbf{u}, p_h) & \leq & C_b \vertiii{\mathbf{u}} \cdot \vertiii{p_h},\,\,\, 
\forall\,\, (\mathbf{u}, p_h) \in(V_*+V_h)\times Q_h, \label{cont2_revised} \\
b_h(\mathbf{u}_h, p) & \leq & C_b \vertiii{\mathbf{u}_h} \cdot \vertiii{p}, \,\,\, \forall\,\, (\mathbf{u}_h, p) \in V_h\times (Q_*+ Q_h). \label{cont4_revised}
\end{eqnarray} 
\end{lem}}

\begin{proof}
The proof is standard and it is omitted for brevity. 
\end{proof} 

\begin{lem}[Stability for $b_h$]\label{stab_b}
There exists $C >0$,  such that for every $p_h \in Q_h$ we have
\begin{equation}\label{b_stab}
C  \left\|p_h\right\|_{\mathrm{\Omega}}\leq  \sup_{\mathbf{w}_h \in V_h\backslash\left\{0\right\}}\frac{b_h(\mathbf{w}_h,p_h)}{\vertiii{\mathbf{w}_h}_{V}}+  k_{\mathcal{T}}(p_h),
\end{equation}
where $k_{\mathcal{T}}(p_h):=\left(\sum_{T\in \mathcal{T}_h} \left\|h_T\nabla p_h\right\|^2_{{\blue T\cap\mathrm{\Omega}}}\right)^{1/2}+
\Big(\sum_{F\in \mathcal{F}_h^{int}}\left\|h_F^{1/2} [\![p_h ]\!]\right\|^2_{F\cap \mathrm{\Omega}}\Big)^{1/2}$.
\end{lem}
\begin{proof}
Consider a fixed $p_h \in Q_h$. Owing to the surjectivity of the divergence operator, there exists a corresponding $\mathbf{v}_{p_h} \in \left[H_0^1(\mathrm{\Omega})\right]^d$, such that 
\begin{equation}\label{est}
\nabla \cdot \mathbf{v}_{p_h} = p_h \quad \textrm{and} \quad C_{\mathrm{\Omega}}\left\|\mathbf{v}_{p_h}\right\|_{1, \mathrm{\Omega}}\leq \left\|p_h\right\|_{\mathrm{\Omega}}
\end{equation}
for some constant $C_{\mathrm{\Omega}}>0$. The field $\mathbf{v}_{p_h}$ is typically referred to as the \itshape{velocity lifting} \normalfont of $p_h$. Then, element--wise integration by parts yields 
\begin{align*}
\left\|p_h\right\|_{\mathrm{\Omega}}^2 
& = \int_{\mathrm{\Omega}}p_h\left(\nabla \cdot \mathbf{v}_{p_h}\right)\,{\dv \mathrm{d}\mathbf{x}}
= -\int_{\mathrm{\Omega}}\mathbf{v}_{p_h}\nabla p_h\, \mathrm{d}\mathbf{x}
+ {\blue \sum_{T \in \mathcal{T}_h}\int_{ \partial T\cap \mathrm{\Omega}} \left(\mathbf{v}_{p_h}\cdot \mathbf{n}_{ T}\right)p_h\,{\mathrm{d}s} +}\\
&\qquad\qquad\qquad\qquad\qquad\qquad\qquad\qquad\qquad\qquad\qquad\quad {\blue + \int_{\mathrm{\Gamma}} \left(\mathbf{v}_{p_h}\cdot \mathbf{n}_{ \mathrm{\Gamma}}\right)p_h\,{\mathrm{d}s}} \\
& = -\int_{\mathrm{\Omega}}\mathbf{v}_{p_h}\nabla p_h\,{ \mathrm{d}\mathbf{x}} 
+ \sum_{F\in  \mathcal{F}_h^{int}}\int_{F\cap \mathrm{\Omega}} \left\{ \mathbf{v}_{p_h}\right\}\cdot \mathbf{n}_F [\![p_h ]\!]\,{\mathrm{d}s}+ \\
&\qquad\qquad\quad + \sum_{F\in \mathcal{F}_h^{int}}\int_{F\cap \mathrm{\Omega}}[\![\mathbf{v}_{p_h}]\!]  \cdot \mathbf{n}_F  \left\{ p_h \right\}\,{ \mathrm{d}s}
+{\blue \sum_{T\in\mathcal{T}_h}\int_{\mathrm{\Gamma}\cap T} \left(\mathbf{v}_{p_h}\cdot \mathbf{n}_{\mathrm{\Gamma}}\right)p_h\,{\mathrm{d}s}}\\
& = -\int_{\mathrm{\Omega}}\mathbf{v}_{p_h}\nabla p_h\,{ \mathrm{d}\mathbf{x}} 
+ \sum_{F\in  \mathcal{F}_h^{int}}\int_{F\cap \mathrm{\Omega}} \left\{ \mathbf{v}_{p_h}\right\}\cdot \mathbf{n}_F [\![p_h ]\!]\,{\mathrm{d}s}.
\end{align*}
Here, we have used the fact that $\mathbf{v}_{p_h}$ and  $ [\![\mathbf{v}_{p_h}]\!]$  vanish on $\mathrm{\Gamma}$ and on $F \in \mathcal{F}_h^{int}$, respectively, due to $\mathbf{v}_{p_h} \in \left[H_0^1(\mathrm{\Omega})\right]^d$ being an element of the continuous space. {\dv Using the vector--valued extended interpolation operator $\mathbf{\Pi}_h: \big[L^2(\mathrm{\Omega})\big]^d\to V_h$ and introducing the corresponding approximation error $\mathbf{e}_h:=\mathbf{\Pi}_h\mathbf{v}_{p_h}-\mathbf{v}_{p_h}$ for $\mathbf{v}_{p_h} \mapsto \mathbf{\Pi}_h \mathbf{v}_{p_h} \in V_h$} in the previous expression, we obtain
\begin{align}
\left\|p_h\right\|_{\mathrm{\Omega}} ^2 
&= \int_{\mathrm{\Omega}}\mathbf{e}_{h}\nabla p_h\,{\blue \mathrm{d}\mathbf{x}}
-\int_{\mathrm{\Omega}}{\dv\mathbf{\Pi}}_h\mathbf{v}_{p_h}\nabla p_h\,{\blue \mathrm{d}\mathbf{x}}
+\sum_{F\in  \mathcal{F}_h^{int}}\int_{F\cap \mathrm{\Omega}} \left\{ \mathbf{v}_{p_h}\right\}\cdot \mathbf{n}_{F} [\![p_h ]\!]\,{\blue \mathrm{d}s} \notag \\
&\stackrel{(\ref{b_alt})}{=}  
\int_{\mathrm{\Omega}}\mathbf{e}_{h}\nabla p_h\, \mathrm{d}\mathbf{x}  
- b_h({\dv\mathbf{\Pi}}_h\mathbf{v}_{p_h}, p_h)- \sum_{F\in  \mathcal{F}_h^{int}}\int_{F\cap \mathrm{\Omega}} \left\{   \mathbf{e}_{h}\right\}\cdot \mathbf{n}_F [\![p_h ]\!]\,{\blue \mathrm{d}s}\nonumber \\
&= \mathbb{I}_1+\mathbb{I}_2+\mathbb{I}_3. \label{rearr}
\end{align}
For the first term, the Cauchy--Schwarz inequality, {\blue the estimates \eqref{interpolation_est1_revised}, \eqref{eq22}} and \eqref{est} imply
\begin{eqnarray}
\left|\mathbb{I}_1\right|&
\leq & \Big(\sum_{T\in \mathcal{T}_h}\left\|h_T^{-1}\mathbf{e}_h\right\|^2_{{\blue T\cap\mathrm{\Omega}}}\Big)^{1/2} \Big(\sum_{T\in \mathcal{T}_h} \left\| h_T \nabla p_h\right\|^2_{{\blue T\cap\mathrm{\Omega}}}\Big)^{1/2} \notag \\
& \lesssim & \left\|\mathbf{v}_{p_h}\right\|_{1,\mathrm{\Omega}}\Big(\sum_{T\in \mathcal{T}_h}\left\|h_T\nabla p_h\right\|^2_{{\blue T\cap\mathrm{\Omega}}}\Big)^{1/2} 
\lesssim C_\mathrm{\Omega}^{-1} \left\|p_h\right\|_{\mathrm{\Omega}}\Big(\sum_{T\in \mathcal{T}_h}\left\|h_T\nabla p_h\right\|^2_{{\blue T\cap\mathrm{\Omega}}}\Big)^{1/2} \label{i1}
\end{eqnarray}
Owing to the continuity property of the extended interpolation operator (\ref{interp_est_revised}) and (\ref{est}) respectively, 
\begin{eqnarray}
\left|\mathbb{I}_2\right| &=& \frac{\left|b_h({\dv\mathbf{\Pi}}_h \mathbf{v}_{p_h}, p_h)\right|}{\vertiii{{\dv\mathbf{\Pi}}_h \mathbf{v}_{p_h}}_{V}}\vertiii{{\dv\mathbf{\Pi}}_h \mathbf{v}_{p_h}}_{V} 
\leq \Big(\sup_{\mathbf{w}_h \in V_h \backslash \left\{0\right\}}\frac{b_h(\mathbf{w}_h, p_h)}{\vertiii{\mathbf{w}_h}_{V}}\Big)C_{\Pi}\left\|\mathbf{v}_{p_h}\right\|_{1, \mathrm{\Omega}}\notag \\
& \leq & \Big(\sup_{\mathbf{w}_h \in V_h \backslash \left\{0\right\}}\frac{b_h(\mathbf{w}_h, p_h)}{\vertiii{\mathbf{w}_h}_{V}}\Big)C_{\Pi}C_{\mathrm{\Omega}}^{-1}\left\|p_h\right\|_{\mathrm{\Omega}}.\label{i2}
\end{eqnarray}
To treat the  third term, we proceed exactly as for $\mathbb{I}_1$ using (\ref{interpolation_est2_revised}) and conclude
\begin{align}
& \left|\mathbb{I}_3\right|
\leq \Big(\sum_{F\in \mathcal{F}_h^{int}}\left\|h_F^{-1/2}\left\{\mathbf{e}_h\right\}\right\|^2_{F\cap \mathrm{\Omega}}\Big)^{1/2} \Big(\sum_{F \in \mathcal{F}_h^{int}} \left\|h_F^{1/2}[\![p_h ]\!]\right\|^2_{F\cap \mathrm{\Omega}}\Big)^{1/2} \notag \\
& \lesssim 
\left\|\mathbf{v}_{p_h}\right\|_{1,\mathrm{\Omega}}\Big(\sum_{F \in \mathcal{F}_h^{int}} \left\|h_F^{1/2}[\![p_h ]\!]\right\|^2_{F\cap \mathrm{\Omega}}\Big)^{1/2}  
\lesssim C_\mathrm{\Omega}^{-1} \left\|p_h\right\|_{\mathrm{\Omega}}\Big(\sum_{F \in \mathcal{F}_h^{int}} \left\|h_F^{1/2}[\![p_h ]\!]\right\|^2_{F\cap \mathrm{\Omega}}\Big)^{1/2}. \label{i3}
\end{align}
Collecting estimates (\ref{i1})-(\ref{i3}), we complete the proof.
\end{proof}

An immediate consequence is the following:

\begin{cor}\label{cornew}
For every $p_h \in Q_h$, there exists $\mathbf{w}_h \in V_h$, such that
\begin{equation}\label{infsup2aa}
b_h(\mathbf{w}_h,-p_h) \geq \left\|p_h\right\|_{\mathrm{\Omega}}^2-C_{\beta} k_{\mathcal{T}}(p_h) \left\|p_h\right\|_{\mathrm{\Omega}},
\end{equation}
for suitable  $C_{\beta}>0$. 
\end{cor}
\begin{proof}
Rearranging (\ref{rearr}), $
b_h({\dv\mathbf{\Pi}}_h \mathbf{v}_{p_h},-p_h) \geq \left\|p_h\right\|_{\mathrm{\Omega}}^2-\left|\mathbb{I}_1\right|-\left|\mathbb{I}_3\right|$. Hence, denoting $C_1$, $C_2$ the constants appearing in (\ref{i1}), (\ref{i3}), the result clearly follows for $\mathbf{w}_h={\dv\mathbf{\Pi}}_h \mathbf{v}_{p_h}$ with $C_{\beta}=\max\left\{C_1, C_2\right\}$.
\end{proof} 

Now we are ready to state the main result of this section.

\begin{thm}[Discrete inf--sup stability]\label{infs}
There is a constant $c_{bil}>0$, such that for all $(\mathbf{u}_h, p_h) \in V_h\times Q_h$, we have
\begin{equation}\label{infsup}
c_{bil}\vertiii{(\mathbf{u}_h, p_h)}_{V,Q}\leq \sup_{(\mathbf{v}_h, q_h)\in V_h\times Q_h}\frac{A_h(\mathbf{u}_h, p_h; \mathbf{v}_h, q_h)+J_h(\mathbf{u}_h, p_h; \mathbf{v}_h, q_h)}{\vertiii{(\mathbf{v}_h, q_h)}_{V,Q}}.
\end{equation}
\end{thm}
\begin{proof} 
Analogous to the ones in \cite[Theorem. 5.1]{MLLR12} and \cite[Theorem 5.3]{BCM15} for unfitted continuous methods. Let $(\mathbf{u}_h, p_h) \in V_h\times Q_h$ and note by Corollary \ref{cornew} that  there exists {\blue $\mathbf{w}_h \in V_h$ satisfying (\ref{infsup2aa})}. In fact, there is no loss of generality in taking $\vertiii{\mathbf{w}_h}_{V}=\left\|p_h\right\|_{\mathrm{\Omega}}$ and then (\ref{infsup2aa}) combined with an $\epsilon$-Young inequality {\blue yields}
\begin{align}
b_h(\mathbf{w}_h,& -p_h) \geq  \left\|p_h\right\|^2_{\mathrm{\Omega}}-C_{\beta} k_{\mathcal{T}}(p_h)\left\|p_h\right\|_{\mathrm{\Omega}} \geq \Big(1-\frac{C_{\beta}\epsilon}{2}\Big)\left\|p_h\right\|^2_{\mathrm{\Omega}}-\frac{C_{\beta}}{2\epsilon}k_{\mathcal{T}}(p_h) ^2\notag \\
\geq \Big(1-&\frac{C_{\beta}\epsilon}{2}\Big)\left\|p_h\right\|^2_{\mathrm{\Omega}}
-\frac{C_{\beta}}{\epsilon}\sum_{T\in \mathcal{T}_h} \left\|h_T\nabla p_h\right\|^2_{T\cap\mathrm{\Omega}}
-\frac{C_{\beta}}{\epsilon}\sum_{F \in \mathcal{F}_h^{int}}\left\| h_F^{1/2} [\![  p_h ]\!]\right\|_{F\cap \mathrm{\Omega}}^2.
\label{b_est_n}
\end{align}
Our purpose is to show that for a judicious choice of parameters $\delta_1>0$ and $\delta_2>0$, there exists a constant $c_{bil}>0$ such that the test pair $(\mathbf{v}_h,q_h)=(\mathbf{u}_h, -p_h) +\delta_1(-\mathbf{w}_h,0)+\delta_2(h^2 \nabla p_h,0)$ satisfies 
\begin{equation}\label{surr}
\left[A_h+J_h\right](\mathbf{u}_h, p_h; \mathbf{v}_h,q_h)\geq c_{bil}\vertiii{(\mathbf{u}_h,p_h)}_{V,Q} \vertiii{(\mathbf{v}_h,q_h)}_{V,Q},
\end{equation}
whereby the assertion (\ref{infsup}) is then immediate.

To this end, if we initially test with $(\mathbf{u}_h,-p_h)$ using the coercivity estimate (\ref{coercivit}) of $\left[a_h+j_u\right]$, we get 
\begin{align}
\big[A_h+J_h\big](\mathbf{u}_h, p_h;&\mathbf{u}_h, -p_h) 
= a_h(\mathbf{u}_h, \mathbf{u}_h) + j_u(\mathbf{u}_h,\mathbf{u}_h) + c_h(p_h, p_h) + j_p(p_h,p_h)\notag  \\ 
& \geq c_{a} \vertiii{\mathbf{u}_h}_{V}^2 + 
\gamma \sum_{F \in \mathcal{F}_h^{int}}\left\| h_F^{1/2} [\![  p_h ]\!]\right\|_{F\cap \mathrm{\Omega}}^2 + j_p(p_h,p_h). \label{partial00}
\end{align}
Next, we consider $(-\mathbf{w}_h,0)$ in \eqref{b_est_n} and apply the continuity estimate (\ref{cont3_revised}) of $\left[a_h+j_u\right]$ along with an $\epsilon$--Young inequality, 
\begin{align}
&\big[A_h+J_h\big](\mathbf{u}_h, p_h;-\mathbf{w}_h, 0) 
= -a_h(\mathbf{u}_h, \mathbf{w}_h)  - j_u(\mathbf{u}_h,\mathbf{w}_h) + b_h(\mathbf{w}_h,-p_h)\notag\\
& \geq -\frac{C_a}{2\epsilon}\vertiii{\mathbf{u}_h}_{V}^2
+\Big(1-\frac{C_{a}\epsilon}{2}-\frac{C_\beta\epsilon}{2}\Big)\left\|p_h\right\|^2_{\mathrm{\Omega}}
-\frac{C_{\beta}}{\epsilon}\sum_{T\in \mathcal{T}_h} \left\|h_T\nabla p_h\right\|^2_{{ T\cap\mathrm{\Omega}}}\nonumber\\
& \qquad\qquad\qquad\qquad\qquad\qquad\qquad\qquad\qquad
\qquad\qquad
-\frac{C_{\beta}}{\epsilon}\sum_{F \in \mathcal{F}_h^{int}}\left\| h_F^{1/2} [\![  p_h ]\!]\right\|_{F\cap \mathrm{\Omega}}^2 \nonumber\\
& \geq -C_1\vertiii{\mathbf{u}_h}_{V}^2
+C_2\left\|p_h\right\|^2_{\mathrm{\Omega}}
-C_3\sum_{T\in \mathcal{T}_h} \left\|h_T\nabla p_h\right\|^2_{{ T\cap\mathrm{\Omega}}}
-C_3\sum_{F \in \mathcal{F}_h^{int}}\left\| h_F^{1/2} [\![  p_h ]\!]\right\|_{F\cap \mathrm{\Omega}}^2,
\label{partial11}
\end{align}
where $C_1=\frac{C_a}{2\epsilon}$, $C_2=1-\frac{C_{a}+C_\beta}{2}\epsilon$ and $C_3=\frac{C_{\beta}}{\epsilon}$ are positive constants for sufficiently small $0<\epsilon<\frac{2}{C_a+C_\beta}$. 

Now, to gain the desired control and compensate over the negative contribution $\|h_T^2\nabla p_h\|_{T\cap\mathrm{\Omega}}^2$ in \eqref{partial11}, we test with $(h^2\nabla p_h,0)$ using the continuity estimate \eqref{cont3_revised} for $\left[a_h+j_u\right]$, the Cauchy-Schwarz inequality, the inverse estimate \eqref{deriv_estimate1} and $\epsilon$--Young inequality in the following fashion: 
\begin{align}
&\big[A_h+J_h\big](\mathbf{u}_h, p_h;h^2\nabla p_h,0) 
= a_h(\mathbf{u}_h, h^2\nabla p_h) + j_u(\mathbf{u}_h, h^2\nabla p_h) + b_h( h^2\nabla p_h,p_h)\notag  \\ 
& \geq - |a_h(\mathbf{u}_h, h^2\nabla p_h) + j_u(\mathbf{u}_h, h^2\nabla p_h)| + \sum_{T\in \mathcal{T}_h}\left\| h_T\nabla p_h \right\|_{T\cap\mathrm{\Omega}}^2 \nonumber \\
& \qquad\qquad\qquad\qquad\qquad\qquad\qquad\qquad\quad
\qquad
-\sum_{F\in\mathcal{F}_h^{int}}\int_{F\cap\mathrm{\Omega}}\{h_F^2\nabla p_h\}\cdot \mathbf{n}_F [\![  p_h ]\!]\,\mathrm{d}s 
\notag \\
& \geq - C_a\vertiii{\mathbf{u}_h}_{V}\vertiii{h^2\nabla p_h}_{V}  + \sum_{T\in \mathcal{T}_h}\left\| h_T\nabla p_h \right\|_{T\cap\mathrm{\Omega}}^2 - \nonumber \\
& \qquad\qquad\qquad - {\blue \Big(\sum_{F\in\mathcal{F}_h^{int}}\left\|h_F^{3/2}\nabla p_h\cdot \mathbf{n}_F\right\|_{F\cap\mathrm{\Omega}}^2\Big)^{1/2} 
\Big(\sum_{F\in\mathcal{F}_h^{int}}\left\|h_F^{1/2}[\![p_h ]\!]\right\|_{F\cap\mathrm{\Omega}}^2\Big)^{1/2}} \nonumber \\
& \geq -\frac{C_a}{2\epsilon_1}\vertiii{\mathbf{u}_h}_V^2 -\frac{C_a\epsilon_1}{2}\vertiii{h^2\nabla p_h}_V^2 
+ \sum_{T\in \mathcal{T}_h}\left\| h_T\nabla p_h \right\|_{T\cap\mathrm{\Omega}}^2 - \nonumber \\
& \qquad\qquad\qquad\qquad\quad\quad\quad
- {\blue C\Big(\sum_{T\in\mathcal{T}_h}\left\|h_T \nabla p_h\right\|_{T}^2\Big)^{1/2} 
\Big(\sum_{F\in\mathcal{F}_h^{int}}\left\|h_F^{1/2}[\![p_h ]\!]\right\|_{F\cap\mathrm{\Omega}}^2\Big)^{1/2}} \nonumber \\
& \geq -\frac{C_a}{2\epsilon_1}\vertiii{\mathbf{u}_h}_V^2 -
{\blue \frac{\widetilde{C}C_a\epsilon_1}{2}\|h\nabla p_h\|_{\mathrm{\Omega}_{\mathcal{T}}}^2} 
+ \sum_{T\in \mathcal{T}_h}\left\| h_T\nabla p_h \right\|_{T\cap\mathrm{\Omega}}^2 - \nonumber \\
& \qquad\qquad\qquad\qquad\qquad\qquad\qquad 
{\blue - \frac{C\epsilon_2}{2} \left\|h\nabla p_h\right\|_{\mathrm{\Omega}_{\mathcal{T}}}^2 - \frac{C}{2\epsilon_2}\sum_{F\in\mathcal{F}_h^{int}}\left\|h_F^{1/2}[\![p_h ]\!]\right\|_{F\cap\mathrm{\Omega}}^2} \nonumber \\
&  \geq -\frac{C_a}{2\epsilon_1}\vertiii{\mathbf{u}_h}_V^2 
+ \sum_{T\in \mathcal{T}_h}\left\| h_T\nabla p_h \right\|_{T\cap\mathrm{\Omega}}^2 
{\blue -  \frac{c_p}{2}(\widetilde{C}C_a\epsilon_1+C\epsilon_2)\Big( \left\|h\nabla p_h\right\|_{\mathrm{\Omega}}^2+j_p(p_h,p_h)\Big) -} \nonumber \\
& \qquad\qquad\qquad\qquad\qquad\qquad\qquad\qquad\qquad
\qquad\qquad 
- {\blue \frac{C}{2\epsilon_2}\sum_{F\in\mathcal{F}_h^{int}}\left\|h_F^{1/2}[\![p_h ]\!]\right\|_{F\cap\mathrm{\Omega}}^2} \nonumber 
\end{align}
\begin{align}
& \geq -C_4\vertiii{\mathbf{u}_h}_V^2 
{\blue + C_5\sum_{T\in \mathcal{T}_h}\left\| h_T\nabla p_h \right\|_{T\cap\mathrm{\Omega}}^2}  
{\blue - C_6 j_p(p_h,p_h)} 
- {\blue C_7}\sum_{F\in\mathcal{F}_h^{int}}\left\|h_F^{1/2}[\![p_h ]\!]\right\|_{F\cap\mathrm{\Omega}}^2 \label{partial22}
\end{align}
where $C_4=\frac{C_a}{2\epsilon_1}$, {\blue $C_5=1-\frac{c_p}{2}(\widetilde{C}C_a\epsilon_1+C\epsilon_2)$, $C_6=\frac{c_p}{2}(\widetilde{C}C_a\epsilon_1+C\epsilon_2)$, and $C_7=\frac{C}{2\epsilon_2}$} are positive constants {\blue for sufficiently small $0<\epsilon_1<2(c_p\widetilde{C}C_a)^{-1}$ and $0<\epsilon_2<2(c_pC)^{-1}\Big(1-\frac{c_p\widetilde{C}C_a\epsilon_1}{2}\Big)$}.

We note that in the {\blue fourth of the above inequalities} and for $\widetilde{C}>0$, we have applied the bound 
\begin{equation}\label{aux_V}
\vertiii{h^2\nabla p_h}^2_{V}= \left\|h^2{\dv \nabla \nabla p_h} \right\|^2_{\mathrm{\Omega}_{\mathcal{T}}}+ \left\|h^{3/2}\nabla p_h\right\|^2_{\mathrm{\Gamma}} + \sum_{F\in \mathcal{F}_{h}^{int}} \left\|h^{3/2}[\![\nabla p_h ]\!]\right\|^2_{F\cap \mathrm{\Omega}}\leq 
\widetilde{C}\|{\blue h \nabla p_h}\|_{\mathrm{\Omega}_{\mathcal{T}}}^2,
\end{equation} 
which has been established by the trace inequalities (\ref{tr1a}), (\ref{tr2a}) and the inverse inequality  (\ref{deriv_estimate3}). In particular, if we regard the norm on a facet $F\subset \partial T\in \mathcal{T}_h$, 
$$
\left\|h^{3/2}\nabla p_h\right\|_{F\cap \mathrm{\Omega}} \leq \left\|h^{3/2}\nabla p_h\right\|_{\partial T} \lesssim h\left\|\nabla p_h\right\|_{T} + h^{2}\left\|{\dv \nabla \nabla p_h}\right\|_{T}\lesssim {\blue h}\left\|{\blue \nabla p_h}\right\|_{T},
$$
by (\ref{tr2a}) and   (\ref{deriv_estimate3}) respectively. Then, the norm corresponding to the jump on $F=T\cap T'$ satisfies
$
\left\|h^{3/2}[\![\nabla p_h ]\!]\right\|_{F \cap \mathrm{\Omega}}\lesssim {\blue h}\max\left\{\left\|{\blue\nabla p_h}\right\|_{T}, \left\|{\blue \nabla p_h}\right\|_{T'}\right\}
$ leading to the estimate $\sum_{F\in \mathcal{F}_{h}^{int}} \left\|h^{3/2}[\![\nabla p_h ]\!]\right\|^2_{F\cap \mathrm{\Omega}}\lesssim  {\blue h}\left\| {\blue \nabla p_h}\right\|_{\mathrm{\Omega}_{\mathcal{T}}}^2$. Proceeding analogously for the other terms, we obtain \eqref{aux_V}.

Overall, we collect inequalities \eqref{partial00}--\eqref{partial22} and we take  
$(\mathbf{v}_h,q_h)=(\mathbf{u}_h, -p_h) +  \delta_1(-\mathbf{w}_h,0)+\delta_2(h^2 \nabla p_h,0)$. Then it holds that  
{\blue
\begin{align}
\big[A_h&+J_h\big](\mathbf{u}_h, p_h;\mathbf{v}_h, q_h) 
\geq 
(c_{a}-\delta_1 C_1 -\delta_2 C_4) \vertiii{\mathbf{u}_h}_{V}^2 
+ \delta_1 C_2\left\|p_h\right\|_{\mathrm{\Omega}}^2 
+ \nonumber \\ 
& \qquad\qquad\qquad\quad 
+(1-\delta_2 C_6)j_p(p_h,p_h)
+(\delta_2 C_5-\delta_1C_3)\sum_{T\in \mathcal{T}_h} \left\|h_T\nabla p_h\right\|^2_{{T\cap\mathrm{\Omega}}} +\nonumber \\
& \qquad\qquad\qquad\qquad\qquad\qquad\qquad\quad
+(\gamma-\delta_1 C_3-\delta_2C_7)\sum_{F\in\mathcal{F}_h^{int}}\left\|h_F^{1/2}[\![p_h ]\!]\right\|_{F\cap\mathrm{\Omega}}^2\nonumber \\
&\geq 
(c_{a}-\delta_1C_1-\delta_2C_4) \vertiii{\mathbf{u}_h}_{V}^2 
+C_p^{-1}\min\{\delta_1 C_2,1-\delta_2 C_6\}\left\|p_h\right\|_{\mathrm{\Omega}_{\mathcal{T}}}^2 +
\nonumber \\ 
& \,\, +(\delta_2C_5-\delta_1C_3)\sum_{T\in \mathcal{T}_h} \left\|h_T\nabla p_h\right\|^2_{{ T\cap\mathrm{\Omega}}} +(\gamma-\delta_1C_3-\delta_2C_7)\sum_{F\in\mathcal{F}_h^{int}}\left\|h_F^{1/2}[\![p_h ]\!]\right\|_{F\cap\mathrm{\Omega}}^2
\label{partial3}
\end{align}}
Finally, if we select {\blue $\delta_2=\frac{2C_3}{C_5}\delta_1>0$} with {\blue $\delta_1<\min\Big\{\frac{c_a}{C_1+2C_3C_4C_5^{-1}}, \frac{C_5}{2C_3C_6}, \frac{\gamma}{C_3(1+2C_7C_5^{-1})}\Big\}$},
then the next inequality follows
\begin{align*}
\big[A_h&+J_h\big](\mathbf{u}_h, p_h;\mathbf{v}_h, q_h) 
\gtrsim \vertiii{\mathbf{u}_h}_V^2 + \vertiii{p_h}_{Q}^2=\vertiii{(\mathbf{u}_h,p_h)}_{V,Q}^2.
\end{align*}
We now note that
\begin{align*}
\vertiii{(\mathbf{u}_h-\delta_1 \mathbf{w}_h+\delta_2 h^2\nabla p_h,-p_h)}_{V,Q}^2 & = 
\vertiii{\mathbf{u}_h-\delta_1 \mathbf{w}_h+\delta_2 h^2\nabla p_h}_{V}^2 +\vertiii{p_h}_{Q}^2 \\
& \leq \vertiii{\mathbf{u}_h}_V^2 +\delta_1 \vertiii{\mathbf{w}_h}_V^2+\delta_2\vertiii{h^2\nabla p_h}_{V}^2 +\vertiii{p_h}_{Q}^2  \\
& \leq \vertiii{\mathbf{u}_h}_V^2 +\Big(\delta_1+{\blue  \frac{2C_3}{C_5}\delta_1}+1\Big) \vertiii{p_h}_{Q}^2  \\
& \leq \Big( 1+{\blue\delta_1\Big(1+\frac{2C_3}{C_5}\Big)}\Big)\vertiii{(\mathbf{u}_h,p_h)}_{V,Q}^2,
\end{align*}
and the result (\ref{surr}) follows for {\blue $c_{bil}=\frac{\min \{\widehat{C}_1,\widehat{C}_2\}}{1+\delta_1(1+2C_3C_5^{-1})}$}, with positive constants {\blue $\widehat{C}_1=c_a-\delta_1(C_1+2C_3C_4C_5^{-1})$ and $\widehat{C}_2=C_p^{-1}\min\{\delta_1C_2,1-2 C_3C_6C_5^{-1}\delta_1\}$.}
\end{proof}
{\blue
\begin{rem}\label{rem1}
In particular, in the proof of Lemma \ref{stab_b} and Theorem \ref{infs} we have used the extended Scott--Zhang interpolation operator for non-smooth functions. An alternative approach would be one to consider an approximate $L^2$--orthogonal projector as in Burman et al. in \cite{BCM15} and to accommodate the above analysis in a similar setting.

\end{rem}
} 

\section{Error estimates} \label{section5}

We first quantify how the additional stabilization form $J_h({\dv \mathbf{u}_h}, p_h ; \mathbf{v}_h, q_h)$ affects the Galerkin orthogonality and consistency of the variational formulation (\ref{cutdg}). {\dv To plug in the exact solution  $(\mathbf{u}, p) \in \left[H^1_0(\mathrm{\Omega})\right]^d \times L_0^2(\mathrm{\Omega})$ into the discrete bilinear form $A_h+J_h$, we extend the domain of $A_h$ and $J_h$ in the following results to a larger product space than $V_h\times Q_h$. Further, to obtain error estimates, in this section we will assume some extra regularity for the solution pair $(\mathbf{u}, p)$}. 

\begin{lem}[Galerkin orthogonality]\label{Galerkin orthogonality}
Let $(\mathbf{u}, p) \in \big[H^2(\mathrm{\Omega})\cap H^1_0(\mathrm{\Omega})\big]^d \times \big[H^1(\mathrm{\Omega})\cap L_0^2(\mathrm{\Omega})\big]$ be the solution to the Stokes problem \normalfont (\ref{The Stokes Problem}) \itshape and $(\mathbf{u}_h, p_h) \in V_h \times Q_h$ the finite element approximation in  \normalfont  (\ref{cutdg}).  \itshape {\dv Assume that the bilinear form $A_h$ is defined on $\big(\big[H^2(\mathrm{\Omega})\cap H^1_0(\mathrm{\Omega})\big]^d +V_h \big)\times \big(\big[H^1(\mathrm{\Omega})\cap L_0^2(\mathrm{\Omega})\big]+Q_h\big)$}. Then, 
\begin{equation}\label{galerkin}
A_h(\mathbf{u}-\mathbf{u}_h, p - p_h; \mathbf{v}_h, q_h) = J_h(\mathbf{u}_h, p_h; \mathbf{v}_h, q_h) \,\,\, \textrm{for every} \,\,  (\mathbf{v}_h, q_h) \in V_h \times Q_h.
\end{equation} 
\end{lem}
\begin{proof}
Recalling the definitions of $A_h$ and $L_h$ in (\ref{A}) -- (\ref{L}) and using the fact that the exact solution $(\mathbf{u}, p)$ satisfies $[\![\mathbf{u}]\!]=[\![  \nabla \mathbf{u}  ]\!] \cdot \mathbf{n}_F=[\![p]\!]=0$ on all interfaces $F \in \mathcal{F}_h^{int}$, we infer using (\ref{a_alt}), (\ref{b_alt}) 
\begin{eqnarray*}
A_h(\mathbf{u}, p ; \mathbf{v}_h, 0) &=& a_h(\mathbf{u}, \mathbf{v}_h) + b_h(\mathbf{v}_h,p) = -\int_{\mathrm{\Omega}}\Delta \mathbf{u} \cdot \mathbf{v}_h\,\,{\blue \mathrm{d}\mathbf{x}}
+\int_{\mathrm{\Omega}}\mathbf{v}_h\cdot \nabla p\,\,{\blue \mathrm{d}\mathbf{x}} \\
A_h(\mathbf{u}, p ; 0, q_h) &=& b_h(\mathbf{u}, q_h)-
{\blue c_h(p, q_h)=0},
\end{eqnarray*}
whereby $A_h(\mathbf{u}, p ; \mathbf{v}_h, q_h)={\blue \int_{\mathrm{\Omega}}\mathbf{f}\cdot \mathbf{v}_h\,\, \mathrm{d}\mathbf{x}=L_h(\mathbf{v}_h)}$  for every $(\mathbf{v}_h, q_h) \in V_h\times Q_h$ and the result follows.
\end{proof}

\begin{lem}[Weak Consistency]\label{consistency_revised}
Let $(\mathbf{u}, p) \in \big[H^{k+1}(\mathrm{\Omega})\cap H^1_0(\mathrm{\Omega})\big]^d \times \big[H^k(\mathrm{\Omega})\cap L_0^2(\mathrm{\Omega})\big]$. {\dv Assume that the bilinear form $J_h$ is defined on $\big(\big[H^{k+1}(\mathrm{\Omega})\cap H^1_0(\mathrm{\Omega})\big]^d +V_h \big)\times \big(\big[H^{k}(\mathrm{\Omega})\cap L_0^2(\mathrm{\Omega})\big]+Q_h\big)$}.
Then, the extended interpolation operator in \normalfont (\ref{interpolation_revised}) \itshape satisfies
\begin{equation}\label{semi_revised}
J_h({\dv\mathbf{\Pi}_h\mathbf{u}, \Pi_h p}; \mathbf{v}_h, q_h) \leq Ch^{k}(|\mathbf{u}|_{k+1, \mathrm{\Omega}}+|p|_{k,\mathrm{\Omega}})\vertiii{(\mathbf{v}_h, q_h)}_{V,Q}
\end{equation} 
\end{lem}
\begin{proof} 
Following the steps in the proof of \cite[Lemma 6.2]{MLLR12} for the appropriate norm $\vertiii{(\cdot, \cdot)}_{V,Q}$, we recall the definition (\ref{j}) of the stabilization term $J_h$,
$$
J_h({\dv\mathbf{\Pi}_h\mathbf{u}, \Pi_h p}; \mathbf{v}_h, q_h)= j_u({\dv\mathbf{\Pi}_h\mathbf{u}}, \mathbf{v}_h)-j_p({\dv\Pi_h p}, q_h).
$$
We first focus on the estimate for the velocity ghost penalty form. Owing to the fact that $\mathbf{u}$  is a continuous function, we have   $j_u({\dv\mathcal{E}^{k+1}}\mathbf{u}, \mathbf{v}_h)=0$. Hence, by (\ref{interpolation_revised})
\begin{align*}
j_u({\dv\mathbf{\Pi}}_h\mathbf{u}, \mathbf{v}_h)& =
j_u({\dv\mathbf{\Pi}_h^* \mathcal{E}^{k+1}}\mathbf{u}-{\dv\mathcal{E}^{k+1}}\mathbf{u}, \mathbf{v}_h)\\  
& \leq \gamma_{\mathbf{u}} \Big(\sum_{F \in \mathcal{F}_G}{\blue\sum_{i=0}^k} h_F^{2i-1} \left\| [\![ \partial_{\mathbf{n}_F}^i\left( {\dv\mathbf{\Pi}_h^* \mathcal{E}^{k+1}}\mathbf{u}-{\dv\mathcal{E}^{k+1}}\mathbf{u}\right)]\!] \right\|^2_{F}\Big)^{1/2}\times \\
& \quad\quad\quad\quad\quad\quad\quad\quad\quad\quad\quad\quad
\times \Big(\sum_{F \in \mathcal{F}_G}{\blue\sum_{i=0}^k} h_F^{2i-1}\left\|[\![\partial_{\mathbf{n}_F}^i \mathbf{v}_h ]\!] \right\|^2_{F}\Big)^{1/2}.
\end{align*} 
To estimate the first factor, we use  the inverse inequalities (\ref{deriv_estimate1}), (\ref{deriv_estimate3}) for a facet $F=T\cap T^{'}$ to obtain 
\begin{align*}
\left\| \partial_{\mathbf{n}_F}^i\left( {\dv\mathbf{\Pi}_h^* \mathcal{E}^{k+1}}\mathbf{u}-{\dv\mathcal{E}^{k+1}}\mathbf{u}\right) \right\|_{F} 
&\lesssim 
h_T^{-1/2}\left\|D^{i}({\dv\mathbf{\Pi}_h^*\mathcal{E}^{k+1}}\mathbf{u}-{\dv\mathcal{E}^{k+1}}\mathbf{u})\right\|_T\\
& \lesssim h_T^{-1/2-i}\left\|{\dv\mathbf{\Pi}_h^*\mathcal{E}^{k+1}}\mathbf{u}-{\dv\mathcal{E}^{k+1}}\mathbf{u}\right\|_T\\ 
&\lesssim
h_T^{k+1-1/2-i}|{\dv\mathcal{E}^{k+1}}\mathbf{u} |_{k+1, \Delta_T}\\
& =h_T^{k+1/2-i}|{\dv\mathcal{E}^{k+1}}\mathbf{u} |_{k+1, \Delta_T}
\end{align*}
and then the related jumps are bounded by 
\begin{eqnarray*}
\left\| [\![  \partial_{\mathbf{n}_F}^i\left( {\dv\mathbf{\Pi}_h^* \mathcal{E}^{k+1}}\mathbf{u}-{\dv\mathcal{E}^{k+1}}\mathbf{u}\right)]\!]  \right\|_{F} 
&\leq &  
\left\| \partial_{\mathbf{n}_F}^i\left( {\dv\mathbf{\Pi}_h^* \mathcal{E}^{k+1}}\mathbf{u}-{\dv\mathcal{E}^{k+1}}\mathbf{u}\right) \right\|_{F\subset T}+ \\
& & \quad\quad\quad\quad +\left\| \partial_{\mathbf{n}_F}^i\left( {\dv\mathbf{\Pi}_h^* \mathcal{E}^{k+1}}\mathbf{u}-{\dv\mathcal{E}^{k+1}}\mathbf{u}\right) \right\|_{F\subset T^{'}}\\ 
&\lesssim & 
h^{k+1/2-i}\Big(|{\dv\mathcal{E}^{k+1}}\mathbf{u} |_{k+1, \Delta_T}+|{\dv\mathcal{E}^{k+1}}\mathbf{u} |_{k+1, \Delta_{T^{'}}}\Big)\\
&\lesssim & h^{k+1/2-i}|{\dv\mathcal{E}^{k+1}}\mathbf{u} |_{k+1, \mathrm{\Omega}_{\mathcal{T}}}.
\end{eqnarray*}
Summing over all $F \in \mathcal{F}_G$, and observing the continuity of $\mathcal{E}$ and the boundedness of the Scott--Zhang interpolation, we have:
\begin{align*}
\Big(\sum_{F \in \mathcal{F}_G}{\blue\sum_{i=0}^k} h_F^{2i-1} \left\| [\![ \partial_{\mathbf{n}_F}^i\left( {\dv\mathbf{\Pi}_h^* \mathcal{E}^{k+1}}\mathbf{u} -{\dv\mathcal{E}^{k+1}}\mathbf{u} \right)]\!] \right\|^2_{F}\Big)^{1/2} 
& \lesssim \Big(\sum_{F \in \mathcal{F}_G}{\blue\sum_{i=0}^k} h^{2k}  |{\dv\mathcal{E}^{k+1}}\mathbf{u}|_{k+1,
\Delta F
}^2 \Big)^{1/2} \\
&\lesssim 
h^{k} |{\dv\mathcal{E}^{k+1}}\mathbf{u}|_{k+1,\mathrm{\Omega}_{\mathcal{T}}}\lesssim 
h^{k} |\mathbf{u}|_{k+1, \mathrm{\Omega}}.
\end{align*}

We proceed similarly for the second factor, noting 
\begin{align*}
\Big(\sum_{F \in \mathcal{F}_G}{\blue\sum_{i=0}^k} h_F^{2i-1}\left\|[\![\partial_{\mathbf{n}_F}^i \mathbf{v}_h ]\!] \right\|^2_{F}\Big)^{1/2}
\lesssim 
\left\| \nabla \mathbf{v}_h\right\|_{\mathrm{\Omega}_{\mathcal{T}}}\lesssim |||\mathbf{v}_h|||_{V}.
\end{align*} 
Hence,
\begin{equation}\label{ju_est}
j_u({\dv\mathbf{\Pi}_h\mathbf{u}}, \mathbf{v}_h) \lesssim h^{k}|\mathbf{u}|_{k+1,\mathrm{\Omega}} |||\mathbf{v}_h|||_{V}.
\end{equation}
For the pressure penalty term, by definition (\ref{jp_high_order}), 
\begin{align*}
j_p({\dv\Pi_h p}, q_h) & \leq \gamma_{p} h 
\Big(\sum_{F \in \mathcal{F}_G}{\blue\sum_{i=0}^k}h_F^{2i-1}\left\|  [\![ \partial_{\mathbf{n}_F}^i ({\dv\Pi_h p}) ]\!] \right\|^2_{F}\Big)^{1/2} \times \\
& \quad\quad\quad\quad\quad\quad \quad\quad\quad \quad\quad\quad 
\times \Big(\sum_{F \in \mathcal{F}_G}{\blue\sum_{i=0}^k} h_T^{2i+1} \left\| [\![ \partial_{\mathbf{n}_F}^i q_h ]\!] \right\|^2_{F}\Big)^{1/2}.
\end{align*}
Following analogue arguments for the first factor, noting the continuity of $p\in H^k(\mathrm{\Omega})\cap L^2_0(\mathrm{\Omega})$, we have $[\![\partial_{\mathbf{n}_F}^ip]\!]=0$ for $i=0, \ldots, k$ and then
\begin{eqnarray*}
{\blue\sum_{i=0}^k}h_F^{2i-1}\left\|  [\![ \partial_{\mathbf{n}_F}^i ({\dv\Pi_h p}) ]\!] \right\|^2_{F}
& = & {\blue\sum_{i=0}^k}h_F^{2i-1}\left\|  [\![ \partial_{\mathbf{n}_F}^i ({\dv\Pi_h p}- p) ]\!] \right\|^2_{F}
\\
& \lesssim & {\blue\sum_{i=0}^k} h^{-2} \left\| {\dv \Pi_h^* {\dv\mathcal{E}^{k}}p}-{\dv\mathcal{E}^{k}}p \right\|^2_{T}
\lesssim h^{2k-2} | {\dv\mathcal{E}^{k}}p|^2_{k,\Delta_{T}},
\end{eqnarray*}
whereby
$$
\Big(\sum_{F \in \mathcal{F}_G}{\blue\sum_{i=0}^k}h_F^{2i-1}\left\|  [\![ \partial_{\mathbf{n}_F}^i ({\dv\Pi_h p}) ]\!] \right\|^2_{F}\Big)^{1/2}
\lesssim 
h^{k-1} | {\dv\mathcal{E}^{k}} p|_{k, \mathrm{\Omega}_{\mathcal{T}}} 
\lesssim h^{k-1} |p|_{k, \mathrm{\Omega}}.
$$
\normalsize
For the second factor, using (\ref{deriv_estimate1}), (\ref{deriv_estimate3}) and summing over all $F \in \mathcal{F}_G$, we conclude the bound
$$
\Big(\sum_{F \in \mathcal{F}_G}{\blue\sum_{i=0}^k} h_T^{2i+1} \left\| [\![ \partial_{\mathbf{n}_F}^i q_h ]\!] \right\|^2_{F}\Big)^{1/2}
\lesssim 
\left\|q_h\right\|_{\mathrm{\Omega}_{\mathcal{T}}}.
$$
Hence, an estimate for the pressure penalty term emerges as 
\begin{equation}\label{jp_est}
j_p({\dv\Pi_h p}, q_h) \lesssim h^{k} |p|_{k, \mathrm{\Omega}}\left\|q_h\right\|_{Q}.
\end{equation}
Combining (\ref{ju_est}) and (\ref{jp_est}), the assertion is immediate. 
\end{proof}

The next result states the main a--priori estimates for the method (\ref{cutdg}). Its proof follows closely the standard arguments with necessary modifications for cut elements; namely, making use of the {\dv extended interpolation operators $\mathbf{\Pi}_h$, $\Pi_h$}  and applying proper cut variants of trace inequalities. It is included here  for completeness. As can be seen, the consistency error in Lemma \ref{consistency_revised} leaves the method's order of convergence unaltered.

\begin{thm}[A--priori error estimate]\label{order_revised}
Let $(\mathbf{u}, p) \in \left[H^{k+1}(\mathrm{\Omega})\cap H^1_0(\mathrm{\Omega})\right]^d \times \left[H^k(\mathrm{\Omega})\cap L_0^2(\mathrm{\Omega})\right]$ be the solution to the Stokes problem \normalfont (\ref{The Stokes Problem}) \itshape and $(\mathbf{u}_h, p_h) \in V_h \times Q_h$ the finite element approximation according to \normalfont (\ref{cutdg}). \itshape {\dv Assume that the bilinear form $A_h$ is defined on $\big(\big[H^{k+1}(\mathrm{\Omega})\cap H^1_0(\mathrm{\Omega})\big]^d +V_h \big)\times \big(\big[H^k(\mathrm{\Omega})\cap L_0^2(\mathrm{\Omega})\big]+Q_h\big)$}. Then, there exists a constant $C>0$, such that
\begin{equation}\label{apriori_revised}
\vertiii{(\mathbf{u}-\mathbf{u}_h, p-p_h)}\leq Ch^{k}(|\mathbf{u}|_{k+1, \mathrm{\Omega}}+|p|_{k,\mathrm{\Omega}}).
\end{equation}
\end{thm}
\begin{proof}
We first decompose the total error $(\mathbf{u}-\mathbf{u}_h, p-p_h)$ into its discrete--error and projection--error components; i.e.,
$$
\vertiii{(\mathbf{u}-\mathbf{u}_h, p-p_h)}\leq \vertiii{(\mathbf{u}-{\dv\mathbf{\Pi}_h\mathbf{u}}, p-{\dv\Pi_h p})}+\vertiii{({\dv\mathbf{\Pi}_h\mathbf{u}}-\mathbf{u}_h, {\dv\Pi_h p}-p_h)}.
$$
Since the desired estimate for the first term is already provided by Corollary \ref{discrete_error_revised}, it clearly suffices to prove the assertion for the latter term, which is in turn  bounded by
$$
\vertiii{({\dv\mathbf{\Pi}_h\mathbf{u}}-\mathbf{u}_h, {\dv\Pi_h p}-p_h)}\leq C\vertiii{({\dv\mathbf{\Pi}_h\mathbf{u}}-\mathbf{u}_h, {\dv\Pi_h p}-p_h)}_{V,Q},
$$
due to (\ref{norm_est1}). To this end, Theorem \ref{infs} ensures the existence of a unit pair $(\mathbf{v}_h, q_h) \in V_h\times Q_h$ with $\left\|(\mathbf{v}_h, q_h)\right\|_{V,Q}=1$, such that 
\begin{align*}
c_{bil}\vertiii{\left({\dv\mathbf{\Pi}_h\mathbf{u}}-\mathbf{u}_h, {\dv\Pi_h p} - p_h\right)}_{V,Q} & \leq  A_h({\dv\mathbf{\Pi}_h\mathbf{u}}-\mathbf{u}_h, {\dv\Pi_h p} - p_h; \mathbf{v}_h, q_h) +\\
& \qquad\qquad\qquad
+ J_h({\dv\mathbf{\Pi}_h\mathbf{u}}-\mathbf{u}_h, {\dv\Pi_h p} - p_h; \mathbf{v}_h, q_h) \\ 
=  A_h(&{\dv\mathbf{\Pi}_h\mathbf{u}}-\mathbf{u}, {\dv\Pi_h p} - p; \mathbf{v}_h, q_h) + J_h(\mathbf{\Pi}_h\mathbf{u}, {\dv\Pi_h p} ; \mathbf{v}_h, q_h),
\end{align*}
where for the last step we invoked the Galerkin orthogonality (\ref{galerkin}) from Lemma \ref{Galerkin orthogonality}. The asserted estimate for the second term follows by Lemma \ref{consistency_revised}, since the pair $(\mathbf{v}_h, q_h) \in V_h\times Q_h$ has unit $\vertiii{(\cdot, \cdot)}_{V,Q}$--norm. Hence, we restrict our attention to the remaining term and use the definition of the corresponding form $A_h$ to express
\begin{eqnarray}
\nonumber 
A_h({\dv\mathbf{\Pi}_h\mathbf{u}}-\mathbf{u}, {\dv\Pi_h p} - p; \mathbf{v}_h, q_h)&= & a_h({\dv\mathbf{\Pi}_h\mathbf{u}}-\mathbf{u}, \mathbf{v}_h) +b_h(\mathbf{v}_h, {\dv\Pi_h p} - p)  +  \\
& & + b_h({\dv\mathbf{\Pi}_h\mathbf{u}}-\mathbf{u},  q_h) + 
{\blue c_h( p- {\dv\Pi_h p}, q_h)}. \label{last}
\end{eqnarray}
{\blue 
The last term in \eqref{last} can be estimated by (\ref{tr2a}), (\ref{interpolation_est2_revised}):
\begin{align*}
c_h(p- \Pi_h p, q_h) &=
\gamma\sum_{F \in \mathcal{F}_h^{int}} \int_{F\cap \mathrm{\Omega}} h_F [\![  p- \Pi_h p ]\!][\![   q_h  ]\!]\,\mathrm{d}s \\
& \leq 
\gamma \Big(\sum_{F \in \mathcal{F}_h^{int}}\left\|h_F ^{1/2}[\![  p- \Pi_h p ]\!]\right\|^2_{F\cap\mathrm{\Omega}}\Big)^{1/2}\Big(\sum_{F \in \mathcal{F}_h^{int}}\left\|h_F^{1/2} [\![ q_h ]\!]\right\|^2_{F\cap\mathrm{\Omega}}\Big)^{1/2} \\
& \lesssim  \gamma h^{k}|p|_{k,\mathrm{\Omega}}\vertiii{q_h}_Q  \lesssim
\gamma h^{k}(|\mathbf{u}|_{k+1,\mathrm{\Omega}}+|p|_{k,\mathrm{\Omega}})\vertiii{q_h}_Q.
\end{align*}
}
Hence, invoking the fact that the pair $\left(\mathbf{v}_h,q_h\right)$ has unit $\vertiii{\cdot}_{V,Q}$--norm, we obtain
$$
c_h(p- {\dv\Pi_h p}, q_h)  \lesssim  h^{k}\left(|\mathbf{u}|_{k+1,\mathrm{\Omega}}+|p|_{k,\mathrm{\Omega}}\right)\vertiii{\left(\mathbf{v}_h,q_h\right)}_{V,Q}= 
h^{k}\left(|\mathbf{u}|_{k+1,\mathrm{\Omega}}+|p|_{k,\mathrm{\Omega}}\right).
$$
In view  of the continuity of $a_h$ and $b_h$ in (\ref{cont1_revised})--(\ref{cont4_revised}) and Corollary \ref{discrete_error_revised}, analogue bounds hold for the remaining terms as well. Hence, an estimate for (\ref{last}) emerges as
\begin{equation*}
A_h({\dv\mathbf{\Pi}_h\mathbf{u}}-\mathbf{u}, {\dv\Pi_h p} - p; \mathbf{v}_h, q_h)\lesssim  h^{k}\left(|\mathbf{u}|_{k+1,\mathrm{\Omega}}+|p|_{k,\mathrm{\Omega}}\right),
\end{equation*}
verifying the validity of (\ref{apriori_revised}).
\end{proof}

\section{Conditioning of the system matrix}\label{conditioning}

Since the inf--sup condition is proved with respect to the $\vertiii{(\cdot, \cdot)}_{V,Q}$--norm, the velocity and the pressure are controlled all over the extended domain $\mathrm{\Omega}_{\mathcal{T}}$. Moreover,  the complete bilinear form $A_h+J_h$ in (\ref{cutdg})  is continuous on discrete spaces in the same norm; see Lemma \ref{cont_a_plus_j} below. Hence, our objective in this section is to verify that the condition number of the matrix of the stabilized unfitted dG formulation (\ref{cutdg}) is uniformly  bounded, independently of how the background mesh $\mathcal{T}_h$ cuts the boundary $\mathrm{\Gamma}$.

\begin{lem}\label{cont_a_plus_j}
There exists a constant $C_{\text{bil}}>0$, such that
\begin{equation}\label{continuity_a_plus_j}
A_h(\mathbf{u}_h, p_h; \mathbf{v}_h, q_h)+J_h(\mathbf{u}_h, p_h; \mathbf{v}_h, q_h) \leq C_{\text{bil}}\vertiii{(\mathbf{u}_h, p_h)}_{V,Q} \vertiii{(\mathbf{v}_h, q_h)}_{V,Q},
\end{equation}
for all $(\mathbf{u}_h, p_h)$, $(\mathbf{v}_h, q_h) \in V_h\times Q_h$.
\end{lem}
\begin{proof}
By the corresponding definitions, we readily obtain 
\begin{align*}
& \left[A_h+J_h\right](\mathbf{u}_h, p_h ; \mathbf{v}_h, q_h)= a_h(\mathbf{u}_h, \mathbf{v}_h) +b_h(\mathbf{u}_h, q_h)+b_h(\mathbf{v}_h, p_h)- {\blue c_h(p_h, q_h)} + \\
& \qquad\qquad\qquad\qquad\qquad\qquad\qquad\qquad\qquad\qquad\qquad\qquad\quad 
+ j_u(\mathbf{u}_h, \mathbf{v}_h)-j_p(p_h, q_h) \\
&\leq C_a \vertiii{ \mathbf{u}_h }_{V} \vertiii{\mathbf{v}_h}_{V} +C_b\vertiii{\mathbf{u}_h}_{V}\vertiii{q_h}_{Q}+C_b\vertiii{\mathbf{v}_h}_{V}\vertiii{p_h}_{Q}-{\blue c_h(p_h, q_h)}-j_p(p_h, q_h) 
\end{align*}
using the continuity estimates (\ref{cont3_revised}), (\ref{cont2_revised}) for $a_h+j_u$ and $b_h$, respectively. For  $c_h(p_h, q_h)$, we proceed as in the  proofs of Theorem \ref{infs} and Theorem \ref{order_revised} to conclude 
{\blue 
\begin{eqnarray*}
|c_h(p_h, q_h)| & \leq & \gamma C_1\vertiii{p_h}_{Q}\vertiii{q_h}_{Q},
\end{eqnarray*}
for some positive constant $C_1$.
}

Similarly, the pressure ghost penalty term
\begin{eqnarray*}
j_p(p_h, q_h) &\leq  j_p(p_h,p_h)^{1/2} j_p(q_h, q_h)^{1/2}  \leq  C_p \left\|p_h\right\|_{\mathrm{\Omega}_{\mathcal{T}}}\left\|q_h\right\|_{\mathrm{\Omega}_{\mathcal{T}}}
\end{eqnarray*}
is controlled by (\ref{ext_p}). Combining all contributions, the result already follows for {\blue
$C_{\text{bil}}=2\max\left\{C_a,C_b, \gamma C_1+C_p\right\}>0$. }


\end{proof}

For our purposes, we will need two auxiliary results. The first  is an inverse estimate for the appropriate norms which will allow us to  bound the discrete energy norm by the $L^2$--norm, while the second is a {\dv discrete} Poincar{\'e}--type inequality which 
{\dv follows analogously to \cite[Proposition 2.12]{GM19}}.

\begin{lem}
There is a constant $C_{\text{inv}}>0$, such that
\begin{equation}\label{inverse_product}
\vertiii{(\mathbf{v}_h, q_h)}_{V,Q}\leq \max\left\{C_{\text{inv}}h^{-1}, 1\right\}\left\|(\mathbf{v}_h, q_h)\right\|_{\mathrm{\Omega}_{\mathcal{T}}},\,\,\,\textrm{for every } (\mathbf{v}_h, q_h) \in V_h\times Q_h,
\end{equation}
{\dv where $\left\|(\mathbf{v}_h, q_h)\right\|_{\mathrm{\Omega}_{\mathcal{T}}}^2=\|\mathbf{v}_h\|_{\mathrm{\Omega}_{\mathcal{T}}}^2+\|q_h\|_{\mathrm{\Omega}_{\mathcal{T}}}^2$.}
\end{lem}
\begin{proof}
We first show the corresponding bound on
\begin{equation}\label{aux}
\vertiii{\mathbf{v}_h}^2_{V}= \left\|\nabla \mathbf{v}_h\right\|^2_{\mathrm{\Omega}_{\mathcal{T}}}+   \left\|h^{-1/2} \mathbf{v}_h\right\|^2_{\mathrm{\Gamma}} + \sum_{F\in \mathcal{F}_{h}^{int}} \left\|h^{-1/2}[\![\mathbf{v}_h ]\!]\right\|^2_{F\cap \mathrm{\Omega}}.
\end{equation} 
All terms are bounded, using the trace inequalities (\ref{tr1a}), (\ref{tr2a}) and the inverse inequality  (\ref{deriv_estimate3}). For instance, regarding the latter term, note for a facet $F\subset T \in \mathcal{T}_h$
$$
\left\|h^{-1/2} \mathbf{v}_h\right\|_{F\cap \mathrm{\Omega}} \leq \left\|h^{-1/2} \mathbf{v}_h\right\|_{\partial T}\lesssim h^{-1}\left\|\mathbf{v}_h\right\|_{T}+h^{1/2}\left\|\nabla\mathbf{v}_h\right\|_{T}\lesssim h^{-1}\left\|\mathbf{v}_h\right\|_{T},
$$
by (\ref{tr2a}) and   (\ref{deriv_estimate3}) respectively. Then, the norm of the corresponding jump on  $F=T\cap T^{'}$ satisfies
$
\left\|{\dv h^{-1/2}}[\![\mathbf{v}_h ]\!]\right\|_{F \cap \mathrm{\Omega}}\lesssim h^{-1}\max\left\{\left\|\mathbf{v}_h\right\|_{T}, \left\|\mathbf{v}_h\right\|_{T^{'}}\right\}
$
and the relevant term in (\ref{aux}) is estimated by $\sum_{F\in \mathcal{F}_{h}^{int}} \left\|h^{-1/2}[\![\mathbf{v}_h ]\!]\right\|^2_{F\cap \mathrm{\Omega}}\lesssim h^{-2}  \left\|\mathbf{v}_h\right\|_{\mathrm{\Omega}_{\mathcal{T}}}^2$. Proceeding in a similar fashion for the first two terms, we obtain the bound
\begin{equation}\label{V_norm}
\vertiii{\mathbf{v}_h}_{V}\leq C_{inv}h^{-1}\left\|\mathbf{v}_h\right\|_{\mathrm{\Omega}_{\mathcal{T}}}
\end{equation}
for some constant $C_{inv}>0$.  Regarding elements in the product space, we conclude by (\ref{V_norm})
$$
\vertiii{(\mathbf{v}_h, q_h)}_{V,Q}^2=\vertiii{\mathbf{v}_h}_{V}^2+\vertiii{q_h}_{Q}^2 \leq\max\left\{1, C_{inv}^2h^{-2}\right\}\left\|(\mathbf{v}_h, q_h)\right\|_{\mathrm{\Omega}_{\mathcal{T}}}^2.
$$ 
\end{proof}

\begin{lem}
There exists a constant $C_{P}>0$, such that
\begin{equation}\label{poincare}
\left\|\mathbf{v}_h\right\|_{\mathrm{\Omega}_{\mathcal{T}}}\leq C_{P}\vertiii{\mathbf{v}_h}_{V},
\end{equation}
for every $\mathbf{v}_h \in V_h$.
\end{lem}
\begin{proof}
Similar to the proof of {\dv\cite[Proposition 2.12]{GM19}}. 
\end{proof}

We are now ready to proceed with the main condition number estimate.

\begin{thm} \label{kAestimate}
The condition number  $\kappa(\mathcal{A})$ of the matrix $\mathcal{A}$ of the stabilized unfitted dG formulation (\ref{cutdg}) satisfies the upper bound
\begin{equation}\label{kAest}
\kappa(\mathcal{A})\leq C_{bil}C_P^2c_{bil}^{-1} \frac{\lambda_{\max}}{\lambda_{\min}} \max\left\{C_{\text{inv}}^2h^{-2}, 1\right\}, 
\end{equation}
where $\lambda_{\min}$ and $\lambda_{\max}$ denote the extreme eigenvalues of the mass matrix $\mathcal{M}$ defined by the bilinear form $\left(\int_{\mathrm{\Omega}_{\mathcal{T}}}\mathbf{u}_h\mathbf{v}_h+\int_{\mathrm{\Omega}_{\mathcal{T}}}p_hq_h\right)$.
\end{thm}
\begin{proof} 
By definition, $\kappa(\mathcal{A}) =\left\|\mathcal{A}\right\|\left\|\mathcal{A}^{-1}\right\|$ and the proof follows by providing appropriate estimates for the operator norms $\left\|\mathcal{A}\right\|$ and $\left\|\mathcal{A}^{-1}\right\|$,  as in \cite[Lemma 11]{BH2012}. For our purposes, since  $\mathcal{T}_h$ is a conforming, quasi--uniform  mesh on the extended domain $\mathrm{\Omega}_{\mathcal{T}}$, we  may use the estimate 
\begin{equation}\label{conforming}
\lambda_{\min}^{1/2} h^{d/2}\left|U\right|_{N}\leq  \left\|\left(\mathbf{u}_h, p_h\right)\right\|_{\mathrm{\Omega}_{\mathcal{T}}}\leq \lambda_{\max}^{1/2} h^{d/2}\left|U\right|_{N},
\end{equation}
to relate the continuous $L^2$--norm of a  finite element function pair $\left(\mathbf{u}_h, p_h\right)$ to the discrete $\ell_2$--norm $|U|_N=(U^TU)^{1/2}$ of the corresponding coefficient vector $U \in \mathbb{R}^N$, where $N=\dim\left(V_h\times Q_h\right)$ and $d \in \left\{2,3\right\}$ is the spatial dimension. To estimate $\left\|\mathcal{A}\right\|$, we let $(\mathbf{u}_h, p_h), (\mathbf{v}_h, q_h) \in V_h \times Q_h$ corresponding to $U, V \in \mathbb{R}^N$ and note that successive application of (\ref{continuity_a_plus_j}), (\ref{inverse_product}) and (\ref{conforming}) yields
\begin{align*}
\left|\mathcal{A}U\right|_{N} & =
\sup_{V \in \mathbb{R}^N}\frac{V^T\mathcal{A}U}{\left| V\right|_{N}}=\sup_{V \in \mathbb{R}^N}\frac{\left[A_h+J_h\right](\mathbf{u}_h, p_h ; \mathbf{v}_h, q_h)}{\left|V\right|_{N}} \\
& \leq C_{bil}\max\left\{C_{\text{inv}}^2h^{-2}, 1\right\}\lambda_{\max}h^{d}|U|_{N},
\end{align*}
whereby $\left\|\mathcal{A}\right\|=\sup_{U \in \mathbb{R}^N}\frac{\left|\mathcal{A}U\right|_{N}}{\left|U\right|_{N}}\leq C_{bil}\max\left\{C_{\text{inv}}^2h^{-2}, 1\right\}\lambda_{\max}h^{d}$.

An estimate for $\left\|\mathcal{A}^{-1}\right\|$ is obtained following a similar procedure. Indeed, letting $U \in \mathbb{R}^N$, Theorem \ref{infs} ensures the existence of a corresponding $V \in \mathbb{R}^N$, such that
$$
V^T\mathcal{A}U= \left[A_h+J_h\right](\mathbf{u}_h, p_h ; \mathbf{v}_h, q_h) \geq c_{bil}  \vertiii{(\mathbf{u}_h, p_h)}_{V, Q} \vertiii{(\mathbf{v}_h, q_h)}_{V,Q} 
$$
and then  successive application of (\ref{poincare}),  (\ref{conforming}) shows that
\begin{equation}\label{oo}
\left|\mathcal{A}U\right|_{N}  = \sup_{W \in \mathbb{R}^N} \frac{W^T\mathcal{A}U }{|W|_{N}}\geq \frac{V^T\mathcal{A}U }{|V|_{N}} \geq c_{bil}  C_P^{-2} \lambda_{\min}h^d |U|_{N}.
\end{equation}
Since $U \in \mathbb{R}^N$ is arbitrary, we may set $V=\mathcal{A}U$  to conclude
\begin{align*}
\left\|\mathcal{A}^{-1}\right\| & =
\sup_{V\in \mathbb{R}^N}\frac{|\mathcal{A}^{-1}V|_{N}}{|V|_{N}}=\sup_{V\in \mathbb{R}^N}\frac{|U|_{N}}{|V|_{N}}\stackrel{(\ref{oo})}{\leq} \sup_{V\in \mathbb{R}^N}\frac{\lambda_{\min}^{-1}C_P^2c_{bil}^{-1}h^{-d}|V|_{N}}{|V|_{N}} \\
& = \lambda_{\min}^{-1}C_P^2c_{bil}^{-1}h^{-d}.
\end{align*}
Combining the estimates for  $\left\|\mathcal{A}\right\|$ and $\left\|\mathcal{A}^{-1}\right\|$ the result already follows.
\end{proof}

\begin{rem}
All constants  in {\dv(\ref{kAest})} are independent of the relative position of the boundary $\mathrm{\Gamma}$ with respect to the background mesh, hence Theorem \ref{kAestimate}  provides a geometrically robust estimate for $\kappa(\mathcal{A})$. For most practical purposes, mesh size $h$ is extremely small and the simplified form 
$$
\kappa(\mathcal{A})\leq C_{bil}C_P^2c_{bil}^{-1} \frac{\lambda_{\max}}{\lambda_{\min}} C_{\text{inv}}^2h^{-2}
$$
of {\dv(\ref{kAest})}  shows that the condition number  can be bounded by $\mathcal{O}(h^{-2})$. 
\end{rem}

\section{Numerical Experiments}  \label{section6}

\subsection{Convergence study}
We consider a two--dimensional test case of (\ref{The Stokes Problem}) in the unit square $\mathrm{\Omega}=\left[0,1\right]^2$ with manufactured exact solution
$$
\mathbf{u}\left(x,y\right)=\left(u\left(x,y\right), -u\left(y,x\right)\right), \quad p\left(x,y\right)=\sin\left(2\pi x\right)\cos\left(2\pi y\right),
$$
where $u(x,y)=\left(\cos\left(2\pi x\right)-1\right)\sin\left(2\pi y\right)$. 
Note that the mean value of $p\left(x,y\right)$ over $\mathrm{\Omega}$ vanishes by construction,
thus ensuring that the problem (\ref{The Stokes Problem}) is uniquely solvable. As in subsection \ref{22}, in the spirit of a fictitious domain approach, we consider the original domain $\mathrm{\Omega}$ as being immersed in the background domain $\mathcal{B}=\left[-0.5, 1.5\right]^{2}$ (see Figure \ref{mesh}).  A level set description of the geometry is possible via the function
\begin{equation}\label{lset}
\phi\left(x,y\right)=\left|x-0.5\right|+\left|y-0.5\right|+\left|\left|x-0.5\right|-\left|y-0.5\right|\right|-1<0.
\end{equation}

To investigate error convergence behavior of the discretization (\ref{cutdg}), we consider a sequence of successively refined tessellations $\{\mathcal{B}_{h_\ell}\}_{\ell>0}$ of $\mathcal{B}$ with
mesh parameters $h_\ell=2^{-\ell-2}$, for $\ell=0, \ldots, 7$. 
{\blue In our implementation, we use equal--order piecewise polynomial spaces of degree $k\in\{1,2,3\}$.  
The discrete inf--sup stability of the proposed pressure--velocity coupling is guaranteed by the stabilizing term $c_h(p_h,q_h)$ in (\ref{c}) which penalizes the pressure jumps across the interior facets of the domain. Moreover, the 
bilinear forms $j_u(\mathbf{u}_h,\mathbf{v}_h)$ in (\ref{ju_high_order}) and $j_p(p_h,q_h)$ in (\ref{jp_high_order}) are essential so as to provide sufficient control over the discrete norms on the whole computational domain. These terms are also critical in order to derive geometrically robust condition numbers and require evaluation of high order normal derivative jumps in the boundary zone for  polynomial degrees up to $k$. 
}

By Theorem \ref{a_coerc}, the symmetric interior penalty parameter $\beta$ in (\ref{a_alt}) should be chosen suitably large for the method to be well-defined, since small values of $\beta$ may affect the quality of the resulting simulation to a great extent increasing both velocity and pressure errors rapidly.
Thus, $\beta$ is judiciously selected to be positively correlated to the finite element order $k$ and to scale as $\beta = 40k^2(k+1)^2$. We note that excessively large values of $\beta$ seem to increase errors, the pressure field error being more sensitive. {\blue In addition, the pressure stabilization parameter $\gamma$ in \eqref{c} also scales in accordance with the polynomial degree $k$, that is $\gamma=10k^2$, and the ghost penalty parameters in (\ref{ju_high_order}), (\ref{jp_high_order})  are chosen as $\gamma_{\mathbf{u}}=\gamma_{p}=\{\gamma_j\}_{j=0}^k=\{40k^2, 0.1, 0.01, 0.001\}$}. Finally, a sparse direct solver has been used to solve the arising linear systems. A sequence of approximations for the first component of the velocity solution in progressively finer unfitted meshes with $k=1$  is illustrated in Figure \ref{coarse}, showcasing the convergence of the method.

As predicted by the theoretical error estimate stated in Theorem \ref{order_revised}, optimal $k$-th order convergence rates with respect to the $H^{1}$--norm of the velocity error and the $L^{2}$--norm of the pressure error are indeed verified by the numerical results in Table \ref{table1} ($k=1$) and Table \ref{table2} ($k=2, 3$), the superiority of the highest order approach being evident.  Indeed, for larger $k$, much smaller errors are attained in progressively smaller mesh sizes. 
{\blue For $k=2$, although initially the pressure convergence rates appear to be low, eventually the expected rates are attained as the mesh becomes finer with the relative pressure errors to decrease}. {\blue We also confirm the effect of the pressure stabilization $c_h(p_h,q_h)$ in (\ref{c}) on discontinuous $P_1-P_1$ elements considering two individual cases in Table \ref{table1}: one stabilized with the term $c_h(p_h,q_h)$  and the other omitting $c_h(p_h,q_h)$, i.e., setting $\gamma=0$. As expected, the stabilization yields better pressure convergence rates, leading to a significant improvement on the pressure errors.

Cases of cut elements in $G_h$ that have an almost zero intersection with the physical domain $\mathrm{\Omega}$ may lead to severe ill conditioning of the system matrix. As mentioned above, this issue is alleviated by penalizing the normal derivative velocity and pressure jumps defined over the fictitious domain across elements that are cut by the unfitted interface. When these terms are suitably regulated by the ghost penalty parameters, then the numerical approach is consistent and geometrically robust. In this context, we devote the following subsection to perform a condition number sensitivity study with respect to cut location.	
}

\begin{table}[htbp]{\blue
\caption{Errors and experimental orders of convergence (EOC) with respect to $H^1$-norm for the velocity and $L^2$-norm for the pressure, using $P_1-P_1$ finite elements. Two cases: one stabilized including the term $c_h(p_h,q_h)$ in (\ref{c}) and the other not stabilized ($\gamma=0$).}\label{table1}  
\resizebox{12cm}{!}{
\begin{tabular}{c|cc|cc|cc|cc}
\toprule
& \multicolumn{2}{c|}{not stabilized}  &  \multicolumn{2}{c|}{stabilized} &  \multicolumn{2}{c|}{not stabilized} &  \multicolumn{2}{c}{stabilized} \\
$h_{\max}$ & $\left\|\mathbf{u}-\mathbf{u}_h\right\|_{1,\mathrm{\Omega}}$ & EOC 
& $\left\|\mathbf{u}-\mathbf{u}_h\right\|_{1,\mathrm{\Omega}}$ & EOC &
$\left\|p-p_h\right\|_{\mathrm{\Omega}}$ & EOC & $\left\|p-p_h\right\|_{\mathrm{\Omega}}$ & EOC\\
\midrule
$2^{-2}$ & 2.38894 &       & 2.35599 &       & 2.42362 &  & 0.79575 &  \\
$2^{-3}$ & 1.07253 & 1.155 & 1.06937 & 1.140 & 2.97449 & -0.296 & 0.48872 & 0.703 \\
$2^{-4}$ & 0.55039 & 0.962 & 0.56680 & 0.916 & 1.59860 & 0.896 & 0.27361 & 0.837 \\
$2^{-5}$ & 0.26797 & 1.038 & 0.28483 & 0.993 & 0.88112 & 0.859 & 0.14629 & 0.903 \\
$2^{-6}$ & 0.13776 & 0.960 & 0.14799 & 0.945 & 0.45211 & 0.963 & 0.07668 & 0.932 \\
$2^{-7}$ & 0.06781 & 1.023 & 0.07301 & 1.019 & 0.22944 & 0.979 & 0.03889 & 0.980 \\
$2^{-8}$ & 0.03381 & 1.004 & 0.03654 & 0.999 & 0.11605 & 0.983 & 0.01966 & 0.984 \\
$2^{-9}$ & 0.01692 & 0.999 & 0.01828 & 0.999 & 0.05780 & 1.005 & 0.00984 & 0.999 \\
\midrule
Mean  &       & 1.020 &     & 1.002 &         & 0.770 &        & 0.905 \\
\bottomrule
\end{tabular}}
}
\end{table}

\begin{table}[htbp]
{\blue
\centering  
\caption{Errors and experimental orders of convergence (EOC) with respect to $H^1$-norm for the velocity and $L^2$-norm for the pressure, using equal order $P_2-P_2$ and $P_3-P_3$ finite elements.}\label{table2}
\resizebox{12cm}{!}{
\begin{tabular}{c|cc|cc|cc|cc} 
\toprule
& \multicolumn{4}{c|}{$P_2-P_2$} &  \multicolumn{4}{c}{$P_3-P_3$}  \vspace{0.1cm} \\ 
$h_{\max}$ & $\left\|\mathbf{u}-\mathbf{u}_h\right\|_{1,\mathrm{\Omega}}$ & EOC 
& $\left\|p-p_h\right\|_{\mathrm{\Omega}}$ & EOC 
& $\left\|\mathbf{u}-\mathbf{u}_h\right\|_{1,\mathrm{\Omega}}$ & EOC
& $\left\|p-p_h\right\|_{\mathrm{\Omega}}$ & EOC \\
\midrule 
$2^{-2}$ & 0.89510 &       & 0.45367 &       & 0.40452 &       & 0.68433 &  \\
$2^{-3}$ & 0.18857 & 2.247 & 0.27160 & 0.740 & 0.02893 & 3.806 & 0.06255 & 3.452 \\
$2^{-4}$ & 0.05428 & 1.797 & 0.19667 & 0.466 & 0.00335 & 3.109 & 0.00490 & 3.674 \\
$2^{-5}$ & 0.01616 & 1.748 & 0.12096 & 0.701 & 0.00062 & 2.438 & 0.00141 & 1.795 \\
$2^{-6}$ & 0.00386 & 2.065 & 0.04889 & 1.307 & 0.00011 & 2.509 & 0.00017 & 3.081 \\
$2^{-7}$ & 0.00088 & 2.141 & 0.01540 & 1.666 & 0.00004 & 1.294 & 0.00007 & 1.242 \\
$2^{-8}$ & 0.00021 & 2.029 & 0.00431 & 1.838 &         &       &  
&  \\
$2^{-9}$ & 0.00008 & 1.623 & 0.00113 & 1.932 &         &       &         &  \\
\midrule
Mean  &       & 1.950 &         & 1.236 &         & 2.631 &         & 2.649 \\
\bottomrule
\end{tabular}}
}
\end{table}

{\blue
\subsection{Condition number tests with respect to the cut location}
The purpose of this subsection is to ascertain the effectiveness of the proposed unfitted dG scheme and its geometric robustness irrespective of the position of the boundary mesh. Therefore, we investigate how the magnitude of the condition number $\kappa(\mathcal{A})$ of the corresponding system matrix $\mathcal{A}$  associated with the unfitted dG formulation \eqref{cutdg} is affected by the position of the boundary with respect to the mesh and the values of the stabilization parameters regulating the ghost--penalty terms.  

To this end, we consider a fixed fictitious domain $\mathcal{B}=[-0.5,1.5]^2$ and a family of immersed  physical domains $\mathrm{\Omega}_{\delta_{\ell}} = [-0.5+\delta_{\ell}, 0.5+\delta_{\ell}]^2$ perturbed with respect to a parameter $\delta_{\ell}=2\ell\cdot 10^{-3}$ for $\ell=1, \ldots, 500$. We use discontinuous $P_k-P_k$ elements of order $k\in\{1, 2, 3\}$ and we construct a quasi-uniform triangulation $\mathcal{T}_h$ with mesh size $h=0.15$. Then we estimate the condition numbers $\kappa(\mathcal{A})$ for each cut configuration corresponding to the polynomial order and plot them against the perturbation parameter. Scaling the symmetric interior penalty constant $\beta=40k^2(k+1)^2$ and the pressure stabilization coefficient $\gamma=10k^2$ with respect to the polynomial degree $k$, we optimize the choice of ghost penalty parameters $\gamma_{\mathbf{u}}= \gamma_p=\{\gamma_{j}\}_{j=0}^k$
among varying values. 

Figures \ref{cond_sensitivity_P1}, \ref{cond_sensitivity_P2} and \ref{cond_sensitivity_P3} overview the aforementioned experiments for discontinuous linear, quadratic and cubic finite elements, respectively. As indicated by the graphs at the top left pictures, variance in $\delta$ may indeed causes severe ill-conditioning dependence on the boundary location if ghost penalty stabilization is removed or only partially activated, i.e. $\gamma_j=0$, $j=0, 1, 2, 3$. Then the condition numbers increase drastically in proportion to the polynomial degree with high oscillatory behavior in relation to $\delta$. As supported by the graphs in Figures \ref{cond_sensitivity_P1} and \ref{cond_sensitivity_P2} (top left pictures), this phenomenon is alleviated for $k=1, 2$ only by taking the full order normal gradient jumps in the stabilization term with coefficients $\gamma_{\mathbf{u}}=\gamma_p=\{40k^2, 0.1, 0.01\}$. However, the numerical evidence for $k=3$ in Figure \ref{cond_sensitivity_P3} illustrate that the effect of ghost penalties seems to decay  with the condition numbers being more sensitive as functions of $\delta$. Nevertheless, if full stabilization $\gamma_{\mathbf{u}}=\gamma_p=\{360, 0.1, 0.01,0.001\}$ is included, then the condition number magnitudes appear to be bounded in a lower values' interval, as expected by Theorem \ref{kAestimate}
. On the other hand, if no stabilization is added, the condition number values 
are reaching up to $10^{31}$ instead of $10^{23}$ for the stabilized case. Some more tests displayed at the right part of Figure \ref{cond_sensitivity_P3} on scaled coefficient $\gamma_0$ with orders of magnitude in the set $\{10^{-2}, 10^{-4},10^{-6}\}$ also convey an unstable behavior with larger spikes than before. 
A possible remedy for this issue would be to pursue the 
techniques analyzed in \cite{L16,S17,S17II,S15} and will be studied in a future work.

Furthermore, the remaining pictures in Figures \ref{cond_sensitivity_P1} and \ref{cond_sensitivity_P2} capture 
the variation of the condition number for $k=1, 2$ over different scaling of the ghost penalty parameters and reveal a lower threshold to produce a robust method. We conduct a number of numerical tests either by simultaneously scaling the parameters $\gamma_{\mathbf{u}}=\gamma_p=\{40k^2, 0.1, 0.01\}$ with orders of magnitude in the set $\{10^{\pm 2}, 10^{\pm 4},10^{\pm 6}\}$ (top right pictures) or by holding one of them fixed at a time (bottom pictures). For $k=2$, we have selected to present the effect of the coefficients $\gamma_1, \gamma_2$ on the condition numbers, since the influence of $\gamma_0$ bears close resemblance to the linear case with only difference the condition numbers magnitudes to range between $10^{10}$ and $10^{15}$. It is clear from the plots that small values of the ghost penalty parameters result in 
condition number instabilities, while excessively large values  lead to large condition numbers. Comparing the observations,  $\gamma_{\mathbf{u}}=\gamma_p=\{40k^2, 0.1, 0.01\}$,  $k\in\{1,2\}$ indicate a fine tuning between the accuracy of the method and the size and fluctuation of the condition number.

\begin{figure}[h]
\includegraphics[scale=0.51]{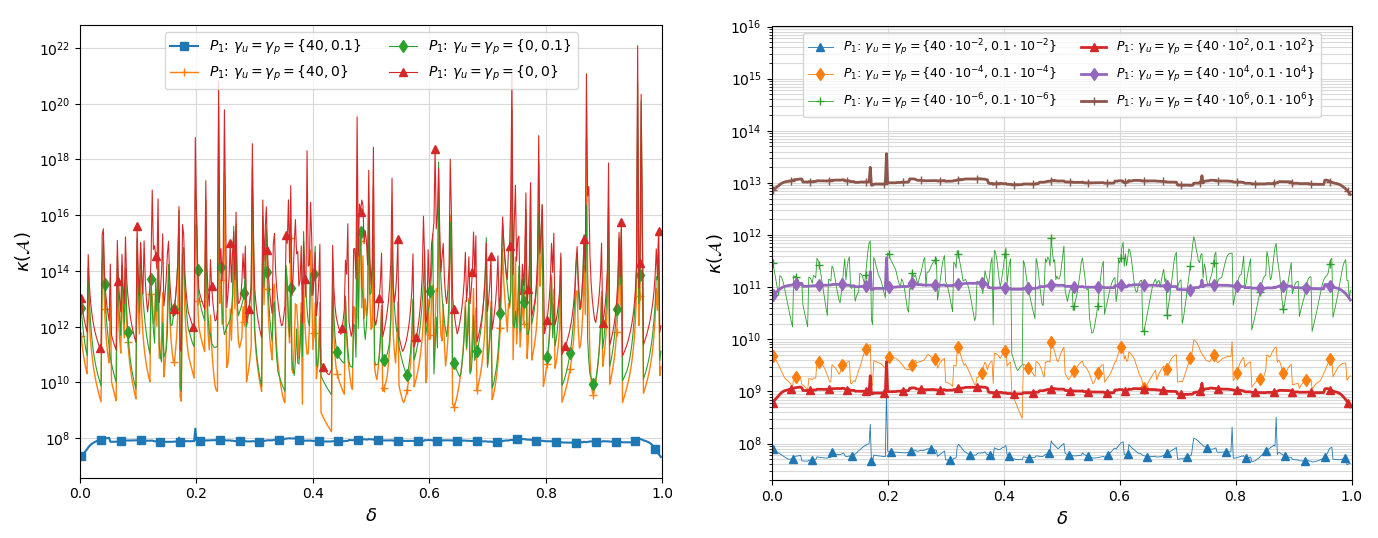}
\!\!\!\!\!\!
\includegraphics[scale=0.5]{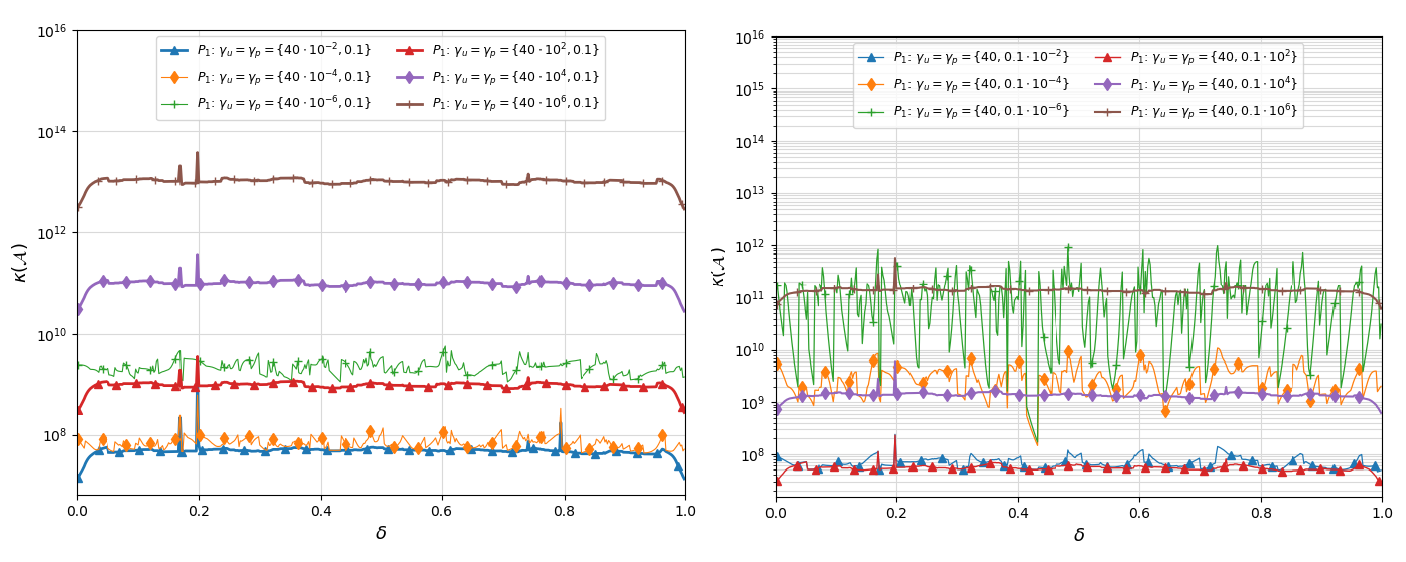}
\caption{{\blue{Condition numbers $\kappa(\mathcal{A})$ with respect to parameter $\delta$ for perturbed physical domains $\mathrm{\Omega}_\delta=[-0.5+\delta, 0.5+\delta]^2$ and varying ghost penalty parameters $\gamma_{\mathbf{u}}=\gamma_{p}$. 
Condition number sensitivity study with and without ghost penalty stabilization (top left) and simultaneously scaled parameters (top right). 
Condition numbers for a variation of coefficients $\gamma_{0}$ (bottom left) and $\gamma_{1}$ (bottom right). All estimates have been computed using discontinuous $P_1-P_1$ elements and symmetric interior penalty parameter $\beta=160$.}}} \label{cond_sensitivity_P1}
\end{figure}

\begin{figure}[h]
\centering
\includegraphics[scale=0.475]{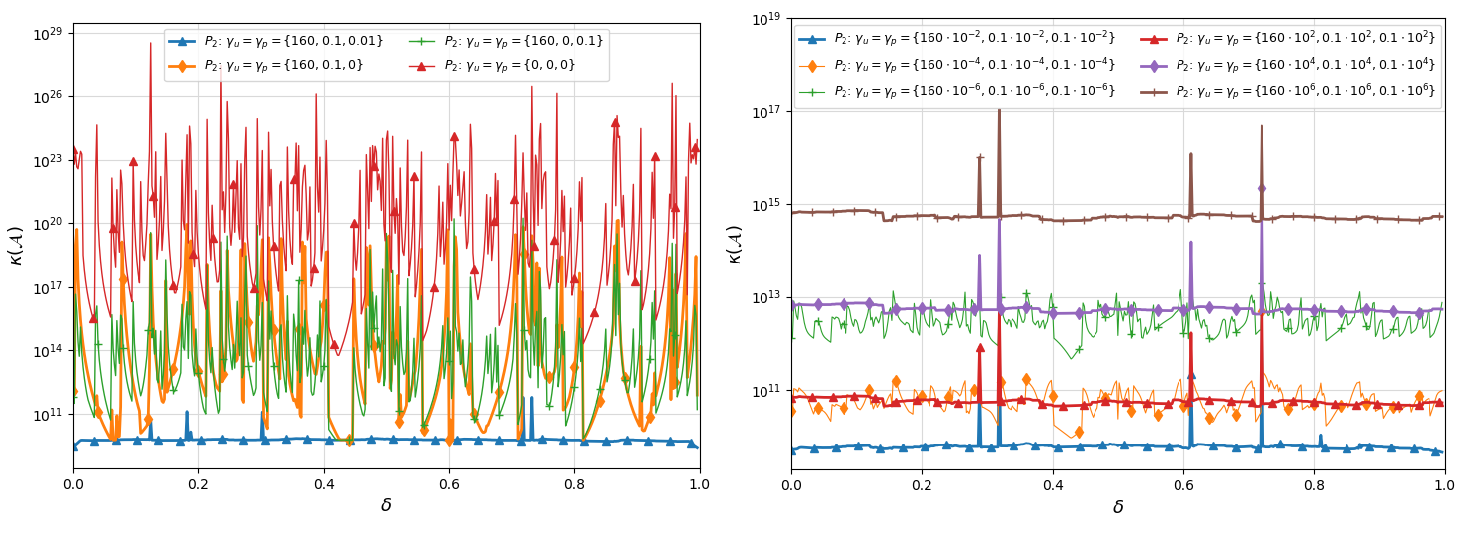}
\includegraphics[scale=0.475]{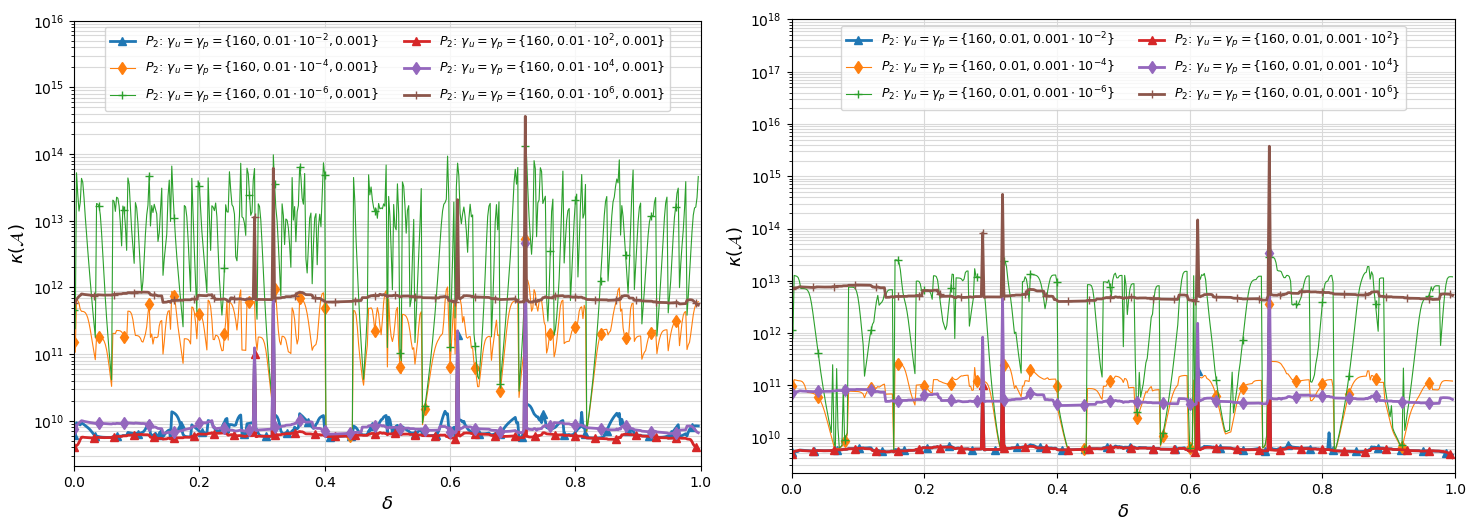}
\caption{{\blue{Condition numbers $\kappa(\mathcal{A})$ with respect to parameter $\delta$ for perturbed physical domains $\mathrm{\Omega}_\delta=[-0.5+\delta, 0.5+\delta]^2$ and varying ghost penalty parameters $\gamma_{\mathbf{u}}=\gamma_{p}$. Condition number sensitivity study with and without ghost penalty stabilization (top left) and simultaneously scaled parameters (top right). Condition numbers for a variation of coefficients $\gamma_{1}$ (bottom left), $\gamma_{2}$ (bottom right). All estimates have been computed using discontinuous $P_2-P_2$ elements and symmetric interior penalty parameter $\beta=1440$.}}} \label{cond_sensitivity_P2}
\end{figure}

\begin{figure}[h]
\includegraphics[scale=0.5]{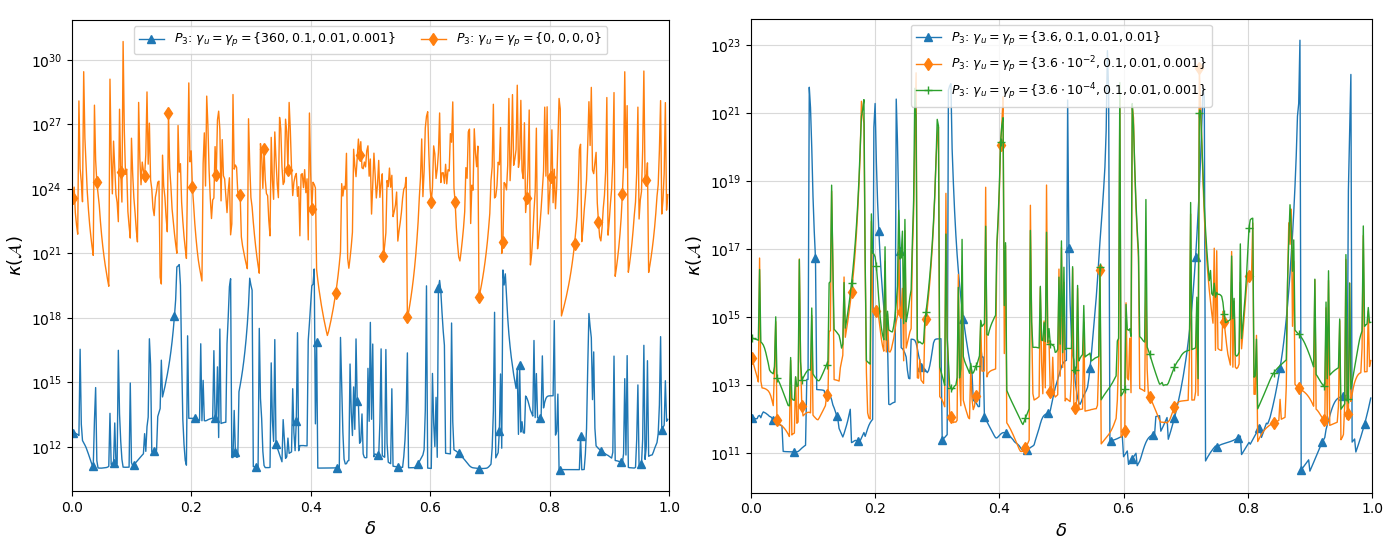}
\caption{{\blue{Condition numbers $\kappa(\mathcal{A})$ with respect to parameter $\delta$ for perturbed physical domains $\mathrm{\Omega}_\delta=[-0.5+\delta, 0.5+\delta]^2$ and varying ghost penalty parameters $\gamma_{\mathbf{u}}=\gamma_{p}$. 
Condition number sensitivity study with and without ghost penalty stabilization (left) and scaled parameter $\gamma_0$ (right). All estimates have been computed using discontinuous $P_3-P_3$ elements and symmetric interior penalty parameter $\beta=1920$.}}} \label{cond_sensitivity_P3}
\end{figure}
}

{\blue
\subsection{Error sensitivity analysis instance for $P_3-P_3$ finite elements with respect to the cut location}

In the current subsection, we focus on the most challenging case of the higher order discontinuous $P_3-P_3$ finite elements to  present the sensitivity of the condition number and the velocity and pressure errors with respect to the cut location. In this experiment, we embed a geometry of smooth boundary representation into a fixed fictitious domain adapting the example of manufactured solution provided in Burman et al \cite{BCM15} to be our exact solution. Let $\mathcal{B}=[-1,1]^2$ be the background domain triangulated with mesh size $h=0.05$. We consider the computational domain to be a circular disk centered at the origin. Compatible with the exact solution, velocity field embedded Dirichlet boundary conditions are also weakly imposed.

In order to produce different cut configurations, we uniformly shift the radius  of the circle  to be $R_{\delta_{\ell}}=R+(\delta_{\ell}+0.04)10^{-2}h$, where $R=0.1$ with respect to a parameter $\delta_{\ell}=2\ell\cdot 10^{-3}$ for $\ell=1, \ldots, 250$. Using discontinuous $P_3-P_3$ finite elements, we  estimate the condition number of the associated stiffness matrix and also we measure the $H^{1}$-norm velocity error and the $L^{2}$-norm pressure error for each cut position. We plot the results against the parameter $\delta$ for successively activated ghost penalty parameters at the optimized values from the previous subsection. 
Representative graphs for the impact of the stabilization term on the condition number and the velocity and pressure errors  are shown in Figures \ref{condP3_circle} and \ref{error_sensitivity_P3}, respectively. 
It is notable that in this test we accomplish robust condition numbers when full ghost penalties are included. 
It is also clearly visible that absence of the stabilization term results in velocity and pressure errors with large spikes and a strong dependence on the location of the interface, whereas the errors become completely insensitive in case of full ghost penalties. A closer look on the instabilities is also provided and highlighted in the zoomed right plots of Figure \ref{error_sensitivity_P3}. 

\begin{figure}[h]
\centering
\includegraphics[scale=0.378]{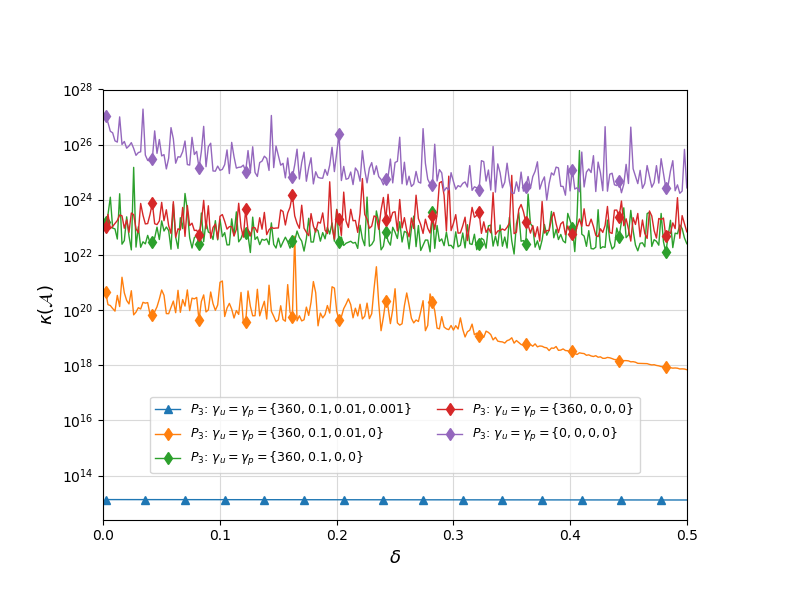}
\caption{{\blue Condition numbers $\kappa(\mathcal{A})$ with respect to parameter $\delta$ for uniformly shifted circular domains of radius $R_{\delta}=0.1+(\delta+0.04)10^{-2}h$ and successively activated ghost penalty parameters $\gamma_{\mathbf{u}}=\gamma_{p}$. The estimates have been computed using discontinuous $P_3-P_3$ elements.}} \label{condP3_circle}
\end{figure}

\begin{figure}[h]
\includegraphics[scale=0.51]{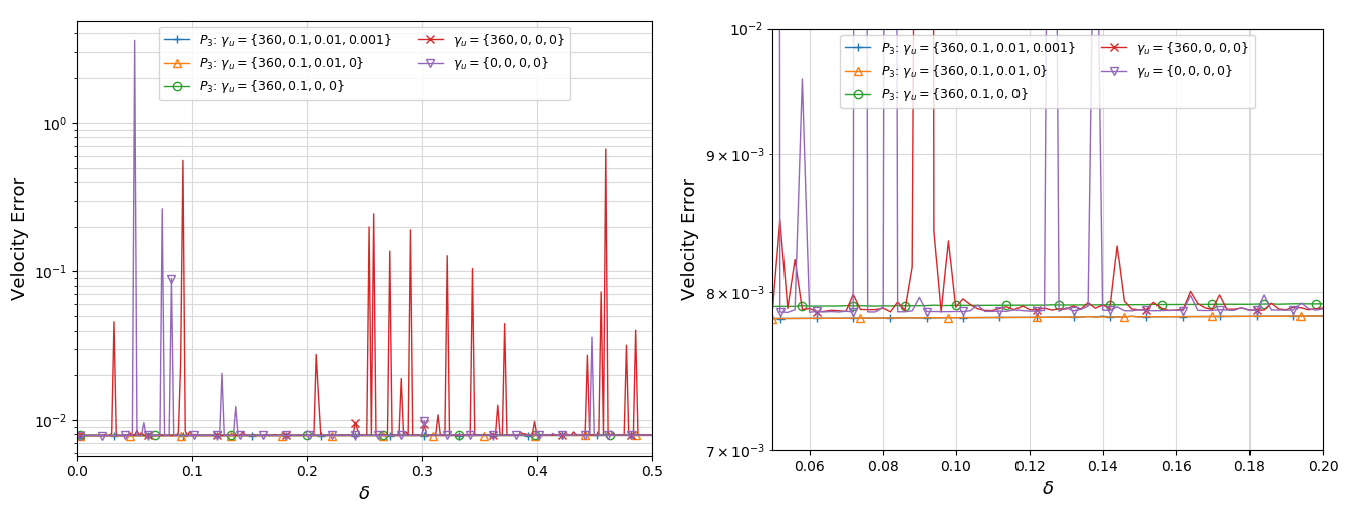}
\!\!\!\!\!\!
\includegraphics[scale=0.521]{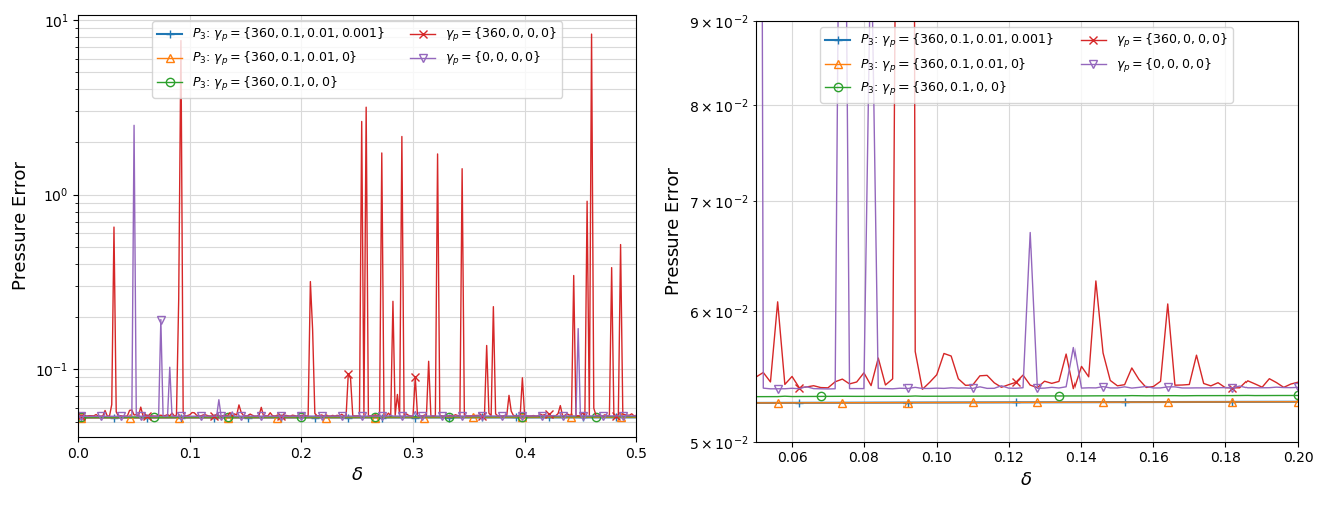}
\caption{{\blue Sensitivities of $H^1$-norm velocity errors (top) and $L^2$-norm pressure errors (bottom) with respect to parameter $\delta$ for uniformly shifted circular domains of radius $R_{\delta}=0.1+(\delta+0.04)10^{-2}h$ and successively activated ghost penalty parameters $\gamma_{\mathbf{u}}=\gamma_{p}$. A closer look on the error instabilities is highlighted at the right pictures. The estimates have been computed using discontinuous $P_3-P_3$ elements.} } \label{error_sensitivity_P3}
\end{figure}

}

\section{Conclusion} 

In this paper, we proposed and tested a stabilized unfitted discontinuous Galerkin method for the incompressible Stokes flow. Optimal order convergence is proved for higher order finite elements which are discontinuous element-wise polynomials of equal order for both velocity and pressure fields. For this equal order case, 
{\blue pressure face jump penalization is employed to achieve stability in the bulk of the domain}. 
Additionally, to ensure stability and error estimates which are independent of the position of the boundary with respect to the mesh, the formulation is augmented with additional boundary zone ghost penalty terms for both velocity and pressure. These terms act on the jumps of the normal   derivatives at faces associated with cut elements. This method may prove valuable in engineering applications where special emphasis is placed on the effective approximation of pressure, attaining much smaller relative errors  in coarser meshes. In fact, control over the error of the pressure field is among the most decisive points of difficulty for many methods. {\blue Additionally, a uniformly bounded estimate for the condition number $\kappa(\mathcal{A})$ of the stiffness matrix is provided}.

Numerical examples demonstrated the stability and accuracy properties of the method. The theoretical  convergence rates for the $H^1$-norm of the velocity and the $L^2$-norm of the pressure have been validated by our tests, even for the $P_3 - P_3 $ case. 
{\blue Finally, 
we employed a condition number sensitivity analysis for a family of perturbed immersed domains {\blue with corners in the embedded boundary} with respect to several cut configurations and a variation of ghost penalty parameters.
In a series of numerical tests for linear and quadratic velocity and pressure approximations, we confirmed the geometrical
robustness of the proposed unfitted dG scheme ensuring a well-conditioned system independent of the location of the interface. However, cubic approximations exhibited condition number oscillations with respect to the cut boundary, their magnitudes though lying in a bounded region of values as expected by the respective theory.} {\blue On the other hand, additional tests on perturbed immersed domains with smooth boundary representation revealed 
robust condition numbers and velocity and pressure errors for the $P_3-P_3$ elements.}

In the present work, we focused on the static Stokes problem. Future work will also extend our investigations to more general fluid mechanics problems, including time--dependent problems on complex and/or evolving domains. 

\section*{Acknowledgements}
This project has received funding from the Hellenic Foundation for Research and Innovation (HFRI) and  the  General  Secretariat  for  Research  and  Technology (GSRT), under  grant agreement No[1115]. This work was supported by computational time granted from the National Infrastructures for Research and Technology S.A. (GRNET S.A.) in the National HPC facility - ARIS - under project ID pa190902 and the “First Call for H.F.R.I. Research Projects to support Faculty members and Researchers and the procurement of high-c                                        ost research equipment” grant 3270. The authors wish to thank Prof. K. Chrysafinos from NTUA and Prof. E.H. Georgoulis from Leicester University for valuable comments and inspiring ideas. {\blue Also the authors would like to express their gratitude to Assoc. Prof. Andr{\'e} Massing and the anonymous referee for the careful reading of the manuscript and for the constructive comments, which helped them improve the quality of this article, as well as, the contributors of the ngsolve \cite{Scho14,ngsolve}, and ngsxfem software \cite{LeHePreWa21,ngsxfem} that have been used.}

\bibliographystyle{amsplain}

\end{document}